\documentclass[11pt]{article}
\usepackage{amsmath,amsthm,amssymb,amsfonts}
\usepackage{latexsym}
\usepackage{graphicx,psfrag,import}
\usepackage{fullpage}
\usepackage{framed}
\usepackage{verbatim}
\usepackage{color}
\usepackage{epsfig}
\usepackage{epstopdf}
\usepackage{hyperref}
\usepackage{geometry}
\usepackage{mathtools}
\usepackage{enumerate}
\usepackage{multicol}
\usepackage{a4wide}
\usepackage{booktabs}
\usepackage{enumitem}
\usepackage{lineno}
\usepackage{parcolumns}
\usepackage{thmtools}
\usepackage{xr}
\usepackage{epstopdf}
\usepackage{mathrsfs}
\usepackage{subfig}
\usepackage{caption}
\usepackage{comment}
\usepackage{authblk}
\usepackage{setspace}

\usepackage{algorithm, algpseudocode}

\makeatletter
\newenvironment{breakablealgorithm}
  {
     \refstepcounter{algorithm}
     \hrule height.8pt depth0pt \kern2pt
     \renewcommand{\caption}[2][\relax]{
       {\raggedright\textbf{\fname@algorithm~\thealgorithm} ##2\par}%
       \ifx\relax##1\relax 
         \addcontentsline{loa}{algorithm}{\protect\numberline{\thealgorithm}##2}%
       \else 
         \addcontentsline{loa}{algorithm}{\protect\numberline{\thealgorithm}##1}%
       \fi
       \kern2pt\hrule\kern2pt
     }
  }{
     \kern2pt\hrule\relax
  }

\theoremstyle{plain}
\newtheorem{theorem}{Theorem}[section]

\newtheorem{proposition}[theorem]{Proposition}
\newtheorem{lemma}[theorem]{Lemma}

\theoremstyle{definition}
\newtheorem{definition}[theorem]{Definition}
\newtheorem{example}[theorem]{Example}

\newtheorem{remark}[theorem]{Remark}

\theoremstyle{remark}

\newcommand\RR{\mathbb{R}}

\newcommand\by{\boldsymbol{y}}
\newcommand\be{\boldsymbol{e}}
\newcommand\bc{\boldsymbol{c}}

\newcommand\bk{\boldsymbol{k}}
\newcommand\bw{\boldsymbol{w}}

\newcommand\bW{\boldsymbol{W}}

\newcommand\bg{\boldsymbol{g}}
\newcommand\bx{\boldsymbol{x}}
\newcommand\bv{\boldsymbol{v}}
\newcommand\bz{\boldsymbol{z}}
\newcommand\bY{\boldsymbol{Y}}
\newcommand\bA{\boldsymbol{A}}

\newcommand\bI{\boldsymbol{I}}

\newcommand\bd{\boldsymbol{d}}
\newcommand\hbd{\hat{\boldsymbol{d}}}
\newcommand{\defi}{\textbf}

\usepackage{tikz, graphicx}
\usepackage[margin=1.5cm]{caption}

	\usetikzlibrary{positioning, arrows, arrows.meta}
	\usepackage{pgfplots}
        \pgfplotsset{compat=1.17}
        \tikzset{%
        fwdrxn/.style={very thick, arrows={-Stealth[length=5pt,width=5pt]}},
        revrxn/.style={very thick, arrows={-Stealth[length=5pt,width=5pt,left]}},
        newt/.style={turq, opacity=0.15}
        }
        \tikzset{near start abs-right/.style={xshift=1cm}}
        \tikzset{near start abs-left/.style={xshift=-3.5cm}}
        \tikzset{near start abs-up/.style={yshift=1.5cm}}
        \tikzset{near start abs-down/.style={yshift=-1cm}}

    \definecolor{viridisyellow}{RGB}{253,231,36}
    \definecolor{viridisyellowpale}{RGB}{239,223,81}
    \definecolor{viridisgreen}{RGB}{121,209,81}
        \definecolor{hlgreen}{RGB}{16,115,16}
    \definecolor{viridisturq}{RGB}{34,167,132}
    \definecolor{viridisblue}{RGB}{64,67,135}
    \definecolor{viridisviolet}{RGB}{68,1,84}

    \definecolor{magmapink}{RGB}{188,81,119}
    \definecolor{pastelpink}{RGB}{253,191,210}
    
	\definecolor{ratecnst}{RGB}{172,172,172}
	
\usepackage{multirow}
\usepackage{array}

\usepackage[normalem]{ulem}	

\begin{document}

\title{Weakly reversible deficiency one realizations of polynomial dynamical systems: an algorithmic perspective}

\author[1]{
         Gheorghe Craciun%
}
\author[2]{
        Abhishek Deshpande%
}
\author[3]{
        Jiaxin Jin%
}
\affil[1]{\small Department of Mathematics and Department of Biomolecular Chemistry, University of Wisconsin-Madison}
\affil[2]{Center for Computational Natural Sciences and Bioinformatics, \protect \\
 International Institute of Information Technology Hyderabad}
\affil[3]{\small Department of Mathematics, The Ohio State University}

\date{}

\maketitle

\begin{abstract}
\noindent
Given a dynamical system with polynomial right-hand side, can it be generated by a reaction network that possesses certain properties? This question is important because some network properties may guarantee specific {\em dynamical} properties, such as existence or uniqueness of equilibria, persistence, permanence, or global stability. Here we focus on this problem in the context of {\em weakly reversible deficiency one networks}. In particular, we describe an algorithm for deciding if a polynomial dynamical system admits a weakly reversible deficiency one realization, and identifying one if it does exist. In addition, we show that weakly reversible deficiency one realizations can be partitioned into mutually exclusive {\em Type I} and {\em Type II} realizations, where Type I realizations guarantee existence and uniqueness of positive steady states, while Type II realizations are related to stoichiometric generators, and therefore to multistability. 
\end{abstract}

\section{Introduction}

By a  \emph{polynomial dynamical system} we mean a dynamical system of the form
\begin{equation}\label{eq:poly-intro}
\begin{split}
    \frac{dx_1}{dt} &= p_1(x_1, \ldots, x_n), \\ 
    \frac{dx_2}{dt} &= p_2(x_1, \ldots, x_n), \\ 
                    &\qquad \quad \vdots  \\
    \frac{dx_n}{dt} &= p_n(x_1, \ldots, x_n), \\ 
\end{split}
\end{equation}
where each $p_i(x_1,\ldots, x_n)$ is a polynomial in the variables  $x_1,\ldots, x_n$. 
Such systems can exhibit exotic behaviors like multistability, presence of oscillations, and chaos due to the underlying nonlinearities.
We are especially interested in the dynamics of these systems when restricted to the {\em positive orthant}, because such systems are very common models of biological interaction networks, population dynamics models, or models for the transmission of infectious diseases~\cite{voit2015150,feinberg2019foundations,gunawardena2003chemical, yu2018mathematical}. 
%
%
%

Very often,  polynomial dynamical systems are generated by {\em reaction networks}. It is often convenient to study the graphical structure of these networks to make inference about their dynamics. It is also possible to study the \emph{inverse problem}, i.e., for some given polynomial dynamical systems, ask what reaction networks can generate them. Due to the phenomenon of \emph{dynamical equivalence}~\cite{craciun2008identifiability,horn1972general}, such a network may not be unique, i.e., there  exist multiple reaction networks that  generate the same dynamics. 

A key quantity in the study of these networks is its {\em deficiency}. In particular, networks possessing low deficiency have been studied in reaction network theory using the Deficiency Zero and Deficiency One theorems~\cite{feinberg2019foundations,horn1972general, horn1972necessary, feinberg1980chemical,feinberg1987chemical,feinberg1995existence}. In particular, the Deficiency One Theorem~\cite{feinberg2019foundations,feinberg1995existence} guarantees  {\em uniqueness} of the steady state within each linear invariant subspace; this, together with the {\em existence} result in~\cite{boros2019existence} completely characterizes the steady states of weakly reversible networks that satisfy the hypotheses of the  Deficiency One Theorem.  Further, weakly reversible networks (i.e., networks where each reaction is part of a cycle) are related to dynamical properties like persistence, permanence, and the existence of a globally attracting steady state~\cite{gopalkrishnan2014geometric, anderson2011proof, craciun2013persistence, boros2020permanence}. 

The problem of identifying weakly reversible deficiency {\em zero} realizations has been addressed in~\cite{craciunalgorithm}. Here we analyze the realizability problem for weakly reversible reaction networks with deficiency {\em one}. In particular, given a polynomial dynamical system, we describe an algorithm to identify if there exists a weakly reversible deficiency one reaction network that generates this dynamical system.
Moreover, if  weakly reversible deficiency one realizations do exist, our algorithm uses the geometry of some convex cones generated using the net reaction vectors to construct one such realization explicitly.

\bigskip

\textbf{Structure of this paper.}
In Section~\ref{sec:reaction_networks}, we recall some basic notions from reaction network theory. 
Primarily, we introduce dynamical equivalence and the matrix of net reaction vectors.
In Section \ref{sec:weakly_reversible_deficiency_one}, we give a short primer on weakly reversible deficiency one reaction networks, and define Type I and Type II weakly reversible deficiency one realizations. 	
In Section~\ref{sec:pointed_cone}, we analyze the pointed cone $\text{ker}(\bW)\cap\mathbb{R}^m_{\geq 0}$ and its minimal set of generators in the context of weakly reversible deficiency one networks.
Moreover, we prove there cannot exist a dynamical equivalence between such networks of two types in Theorem~\ref{thm:unique of deficiency_one under same linkage}. 
In Section~\ref{sec:algorithms}, we state the main algorithm of our paper: Algorithm \ref{algorithm:WR_def_one}, which checks the existence of a weakly reversible deficiency one realization and returns a realization if it exists.
In Section~\ref{sec:discussion}, we summarize our findings in this paper and flesh out directions for future work.

\bigskip

\textbf{Notation.}
We let $\mathbb{R}_{\geq 0}^n$ and $\mathbb{R}_{>0}^n$ denote the set of vectors in $\mathbb{R}^n$ with non-negative and positive entries respectively. 
Given two vectors $\bx = (\bx_1, \ldots, \bx_n)^{\intercal}\in \RR^n_{>0}$ and $\by = (\by_1, \ldots, \by_n)^{\intercal} \in \RR^n$, we use the vector operation as follows:
\begin{equation} \notag
\bx^{\by} := \bx_1^{y_{1}} \ldots \bx_n^{y_{n}}.
\end{equation}
Given a positive integer $m$, we denote $[m] := \{1, \ldots, m \}$.

\section{Reaction networks}
\label{sec:reaction_networks}

\subsection{Terminology}

\begin{definition} 
A \defi{reaction network} $G = (V, E)$, also called the \defi{Euclidean embedded graph (E-graph)}, is a directed graph in $\RR^n$, where $V \subsetneq \mathbb{R}^n$ represents a finite set of \defi{vertices}, and $E \subseteq V \times V$ represents a finite set of \defi{edges}. 
In this paper, there are neither self-loops nor isolated vertices in $G$.

\begin{enumerate}[label=(\alph*)]
\item Let $V = \{ \by_1, \ldots, \by_m \}$, and denote the \textbf{number of vertices} in $G$ by $m$.

\item We denote a directed edge by $(\by_i, \by_j) \in E$, which represents a reaction in the network. Here $\by_i$ and $\by_j$ are called the \defi{source vertex} and \defi{target vertex} respectively.
Moreover, we denote the \defi{reaction vector} associated with the edge $\by_i \rightarrow \by_j$ by $\by_j - \by_i \in\mathbb{R}^n$. 
\end{enumerate}
\end{definition}

\begin{definition}
Let $G=(V, E)$ be a Euclidean embedded graph.
The \defi{stoichiometric subspace of $G$} is the vector space spanned by the reaction vectors as follows:
\begin{equation} \notag
S = \text{span} \{\by' - \by\, |\, \by \rightarrow \by' \in E \}.
\end{equation}
Given a subset of vertices $V_0 \subseteq V$, the \defi{stoichiometric subspace defined by $V_0$} is
\begin{equation} \notag
S(V_0) = \text{span} \{ \by' - \by\, |\, \by \rightarrow \by' \in E \ \text{and} \ \by', \by \in V_0 \}.
\end{equation}
Furthermore, given a positive vector $\bx_0 \in\mathbb{R}_{>0}^n$, the polyhedron $(\bx_0 + S ) \cap \mathbb{R}^n_{>0 }$ is called the \defi{stoichiometric compatibility class} of $\bx_0$.
\end{definition} 

\begin{definition} 
Let $G=(V, E)$ be a Euclidean embedded graph.

\begin{enumerate}[label=(\alph*)]
\item The set of vertices $V$ is partitioned by its connected components, also called \defi{linkage classes}.
Every connected component is denoted by the set of vertices belonging to it. 

\item A connected component $L \subseteq V$ is \defi{strongly connected}, if every edge is part of an oriented cycle. 
Moreover, a strongly connected component $L \subseteq V$ is \defi{terminal strongly connected}, if for every vertex $\by\in L$ and $\by\rightarrow\by'\in E$, we have $\by'\in L$.

\item $G=(V,E)$ is \defi{weakly reversible}, if all connected components are strongly connected.
\end{enumerate}
\end{definition}

\begin{remark}
For any weakly reversible reaction network $G=(V, E)$, every vertex $\by \in V$ is a source and a target vertex.
Moreover, every connected component of $G$ is strongly connected, and terminal strongly connected.
\end{remark}

\begin{definition}
Let $G = (V, E)$ be a Euclidean embedded graph, which contains $m$ vertices and $\ell$ connected components. Denote the dimension of the stoichiometric subspace of $G$ by $s = \dim (S)$, the \defi{deficiency} of $G$ is a non-negative integer as follows:
\begin{equation*}    
\delta = m - \ell - s.
\end{equation*}
\end{definition}

\medskip

\noindent When considering a linkage class with $V_i \subseteq V$, we define the \defi{deficiency of a linkage class} as 
\begin{equation*}
\delta_i = |V_i| - 1 - \dim S(V_i).
\end{equation*}

One can check that 
\begin{equation} \label{delta inequality}
\delta \geq \sum_{i=1}^\ell \delta_i,
\end{equation}
where the equality holds when the stoichiometric subspaces of all linkage classes are linearly independent.

\medskip

\begin{figure}[H]
\centering
\includegraphics[scale=0.5]{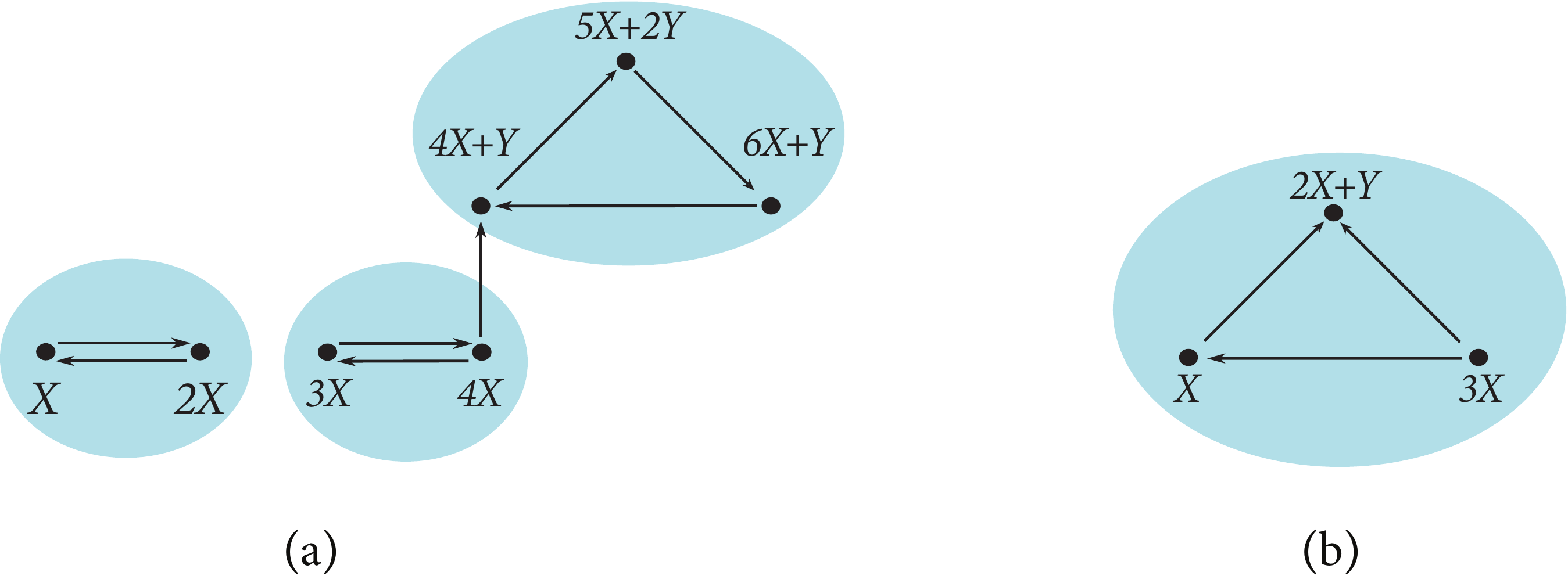}
\caption{\small (a) This reaction network consists of two linkage classes and contains three terminal strongly connected components (shown in circles).  It has a stoichiometric subspace of dimension 2 and a deficiency $\delta = n - \ell -s = 7- 2-2=3$. (b) This reaction network is weakly reversible and contains one terminal strongly connected component. It has a stoichiometric subspace of dimension 2 and a deficiency $\delta = n - \ell -s = 3- 1-2=0$.}
\label{fig:euclidean_graphs}
\end{figure} 

Figure~\ref{fig:euclidean_graphs} shows two reaction networks represented as Euclidean embedded graphs.
Given a reaction network, it can generate an extensive variety  of dynamical systems. Here, we focus on mass-action kinetics, which has been studied in~\cite{voit2015150,gunawardena2003chemical,yu2018mathematical,guldberg1864studies,feinberg1979lectures,adleman2014mathematics}.

\begin{definition} 
Let $G=(V,E)$ be a Euclidean embedded graph, we 
let $\bk = (k_{\by_i \rightarrow \by_j})_{\by_i \rightarrow \by_j \in E} \in \mathbb{R}_{>0}^{E}$
denote the \defi{vector of reaction rate constants}, where $k_{\by_i \rightarrow \by_j}$ or $k_{ij}$
is called the \defi{reaction rate constant} of the edge $\by_i \rightarrow \by_j \in E$. Furthermore, the \defi{associated mass-action system} generated by $(G, \bk)$ 
on $\RR^n_{>0}$ is
\begin{equation} \label{eq:mass_action}
\frac{\mathrm{d} \bx}{\mathrm{d} t}= \sum_{\by_i \rightarrow \by_j \in E}k_{\by_i \rightarrow \by_j} \bx^{\by_i}(\by_j - \by_i).
\end{equation}
A point $\bx^* \in \mathbb{R}_{>0}^n$ is called a \defi{positive steady state}, if 
\begin{equation} \label{eq:ss}
\sum_{\by_i \rightarrow \by_j \in E}k_{\by_i \rightarrow \by_j} (\bx^*)^{\by_i}(\by_j - \by_i)
= \mathbf{0}.
\end{equation}
\end{definition}

\noindent From~\cite{horn1972general}, it is known that every mass-action system admits the following matrix decomposition:
\begin{equation} \label{eq:mass_action_horn_jackson}
\frac{d\bx}{dt} = \bY \bA_{\bk}{\bx}^{\bY},
\end{equation}
where $\bY$ is called the \defi{matrix of vertices}, whose columns are the vertices, defined as
\begin{equation*}
\bY = (\by_1, \ \by_2, \ \ldots, \ \by_m),
\end{equation*}
 ${\bx}^{\bY}$ is the vector
of monomials given by
\begin{equation*}
{\bx}^{\bY} = ({\bx}^{\by_1},{\bx}^{\by_2}, \ldots, {\bx}^{\by_m})^{\intercal},
\end{equation*} 
and $\bA_{\bk}$ is the negative transpose of the Laplacian of $(G, \bk)$, defined as
\begin{equation*}
[\bA_{\bk}]_{ji} =
\begin{cases}
k_{\by_i\rightarrow\by_j}, & \text{if } \by_i\rightarrow\by_j \in E, \\[5pt]
- \sum\limits_{\by_i\rightarrow\by_j \in E} k_{\by_i\rightarrow\by_j}, & \text{if } i=j, \\[5pt]
0,  & \text{otherwise}.
\end{cases}
\end{equation*}
Here, $\bA_{\bk}$ is called the \defi{Kirchoff} matrix, whose column sums are zero from the definition. 

\medskip

The properties of the kernel of $\bA_{\bk}$ are well known in reaction network theory~\cite{feinberg2019foundations,gunawardena2003chemical,feinberg1977chemical}. 
Below we collect some of the most important properties.

\begin{theorem}[\cite{feinberg1977chemical}]
\label{thm:supp_terminal_linkage}
Let $(G, \bk)$ be a mass-action system, and $T_1, T_2, \ldots, T_t$ be the terminal strongly connected components of $G$. Then there exists a basis $\{\bc_1, \bc_2, \ldots, \bc_t\}$ for $\ker (\bA_{\bk})$, such that 
\begin{equation*}
\bc_q =
\begin{cases} 
    \begin{array}{cl}
         [\bc_q]_i  > 0, & \text{ if } \by_i \in T_q, \\[5pt]
         [\bc_q]_i  = 0, & \text{ otherwise.}
    \end{array} 
\end{cases}
\end{equation*}
\end{theorem}

\begin{proposition}[{\cite[Corollary 4.2]{feinberg1979lectures}}] \label{prop:positive_vector_kernel}
Consider a mass-action system $(G, \bk)$, then $G$ is weakly reversible if and only if the kernel of the Kirchoff matrix $\bA_{\bk}$ contains a positive vector. 
\end{proposition}

\subsection{Net reaction vectors and dynamical equivalence}
\label{subsec:net_vectors}

Inspired by the matrix decomposition in \eqref{eq:mass_action_horn_jackson}, we introduce the key concept: net reaction vector, and illustrate a new matrix decomposition in terms of net reaction vectors.

\begin{definition} \label{defn:net_reaction_vector}
Let $(G, \bk)$ be a mass-action system, and $V_s = \{ \by_1, \by_2, \ldots, \by_{m_s}\} \subseteq V$ be \defi{the set of source vertices of $G$}. For each source vertex $\by_i \in V_s$, we denote the \defi{net reaction vector} $\bw_i$ corresponding to $\by_i$ by
\begin{equation} 
\bw_i = \sum\limits_{\by_i\rightarrow\by_j\in E}k_{\by_i\rightarrow\by_j}(\by_j - \by_i).
\end{equation}
Further, we denote the \defi{matrix of net reaction vectors of $G$} as follows:
\begin{equation}
\bW = \left(\bw_1, \bw_2, \ldots, \bw_{m_s} \right).
\end{equation}
It is convenient to refer to $\bw_j$ even when $\by_j \not\in V_s$, in which case we consider $\bw_j$ represents an empty sum, i.e., $\bw_j = \mathbf{0}$.
\end{definition}

\noindent From Definition \ref{defn:net_reaction_vector}, every net reaction vector $\bw_i$ corresponding to $\by_i$ can be expressed as 
\begin{equation}
\bw_i = \sum\limits_{\by_i\rightarrow\by_j\in E}k_{\by_i\rightarrow\by_j} \by_j 
- \bigg( \sum\limits_{\by_i\rightarrow\by_j\in E}k_{\by_i\rightarrow\by_j} \bigg) \by_i.
\end{equation}
Using a direct computation, we rewrite the mass-action system in~\eqref{eq:mass_action} as
\begin{equation} \label{eq:mass_action_net_reaction}
\frac{d\bx}{dt} = \bW {\bx}^{\bY_s},
\end{equation}
where $\bY_s$ is called the \defi{matrix of source vertices}, whose columns are source vertices, defined as
\begin{equation*}
\bY_s = (\by_1, \by_2, \ldots, \by_{m_s}),
\end{equation*}
and ${\bx}^{\bY_s}$ is the vector
of monomials given by
\begin{equation*}
{\bx}^{\bY_s} = ({\bx}^{\by_1},{\bx}^{\by_2}, \ldots, {\bx}^{\by_{m_s}})^{\intercal}.
\end{equation*}

Note that for any weakly reversible mass-action system $(G, \bk)$, we have $V_s = V$ and $\bY_s = \bY$.
Moreover, we derive that $\bW = \bY \bA_{\bk}$, which follows from matrix decomposition in \eqref{eq:mass_action_horn_jackson}.

\begin{definition}
Let $(G,\bk)$ and $(\bar{G},\bar{\bk})$ be two mass-action systems. Then $(G,\bk)$ and $(\bar{G},\bar{\bk})$ are called \defi{dynamically equivalent}, if for any $\bx \in \RR^n_{>0}$,
\begin{equation} \label{eq:eqDE}
\sum\limits_{\by\rightarrow \by'\in E}k_{\by\rightarrow \by'}{\bx}^{\by}(\by' - \by) =  \sum\limits_{\bar{\by}\rightarrow \bar{\by}'\in \bar{E}}\bar{k}_{\bar{\by}\rightarrow \bar{\by}'}{\bx}^{\bar{\by}}(\bar{\by}' - \bar{\by}).
\end{equation}
\end{definition}

\begin{remark} \label{rmk:dyn_net_vectors}
From Equation~\eqref{eq:eqDE}, we achieve a necessary and sufficient condition for dynamical equivalence 
between $(G,\bk)$ and $(\bar{G},\bar{\bk})$: 
for every vertex $\by_0 \in V_s \cup \bar{V}_s$, 
\begin{equation} \label{eq:DE}
\sum_{\by_0 \rightarrow \by \in E} k_{\by_0 \rightarrow \by} (\by - \by_0) =
\sum_{\by_0 \rightarrow \by' \in \bar{E}} \bar{k}_{\by_0 \rightarrow \by'}  (\by' - \by_0),
\end{equation}
From Definition \ref{defn:net_reaction_vector}, this is equivalent to
\begin{equation}
    \bw_0 =  \Bar{\bw}_0.
\end{equation}
Note that if either $\by_0 \not\in V_s$ or $\by_0 \not\in \bar{V}_s$, then one side of Equation~\eqref{eq:DE} gives an empty sum, i.e., $\bw_0 = \mathbf{0}$ or $\bar{\bw}_0 = \mathbf{0}$.
\end{remark}

\begin{example}
Two dynamically equivalent mass-action systems are presented in Figure~\ref{fig:dynamical_equivalence}. The mass-action systems (a) $(G, \bk)$ and (b) $(G', \bk')$ share the vertices
\begin{equation} \notag
\by_1 = \begin{pmatrix} 0 \\ 0 \end{pmatrix}, \quad 
\by_2 = \begin{pmatrix} 0 \\ 2 \end{pmatrix}, \quad 
\by_3 = \begin{pmatrix} 2 \\ 0 \end{pmatrix}.
\end{equation}
The reaction network $G'$ has an additional vertex
\begin{equation} \notag
\by_4 = \begin{pmatrix}  1 \\ 1 \end{pmatrix}.
\end{equation}
Given the rate constants in Figure \ref{fig:dynamical_equivalence}, we note that $\by_1$ is the only source vertex in both $G$ and $G'$. 
Thus, it suffices to check whether two systems satisfy Equation~\eqref{eq:DE} on the vertex $\by_1$.

For the system $(G, \bk)$, we have
\begin{equation}
\sum_{\by_1 \rightarrow \by \in E} k_{\by_1 \rightarrow \by} (\by - \by_1) =
k_{12} \begin{pmatrix} 0 \\ 2 \end{pmatrix}
+ k_{13} \begin{pmatrix} 2 \\ 0 \end{pmatrix}
= \begin{pmatrix} 4 \\ 4 \end{pmatrix}.
\end{equation}

For the system $(G', \bk')$, we have
\begin{equation}
\sum_{\by_1 \rightarrow \by' \in E'} k_{\by_1 \rightarrow \by'} (\by' - \by_1) =
k'_{12} \begin{pmatrix} 0 \\ 2 \end{pmatrix}
+ k'_{13} \begin{pmatrix} 2 \\ 0 \end{pmatrix}
+ k'_{14} \begin{pmatrix} 1 \\ 1 \end{pmatrix}
= \begin{pmatrix} 4 \\ 4 \end{pmatrix}.
\end{equation}
This shows two systems have the same net reaction vector corresponding to the source vertex $\by_1$, and are hence dynamically equivalent.

\begin{figure}[H]
\centering
\includegraphics[scale=0.5]{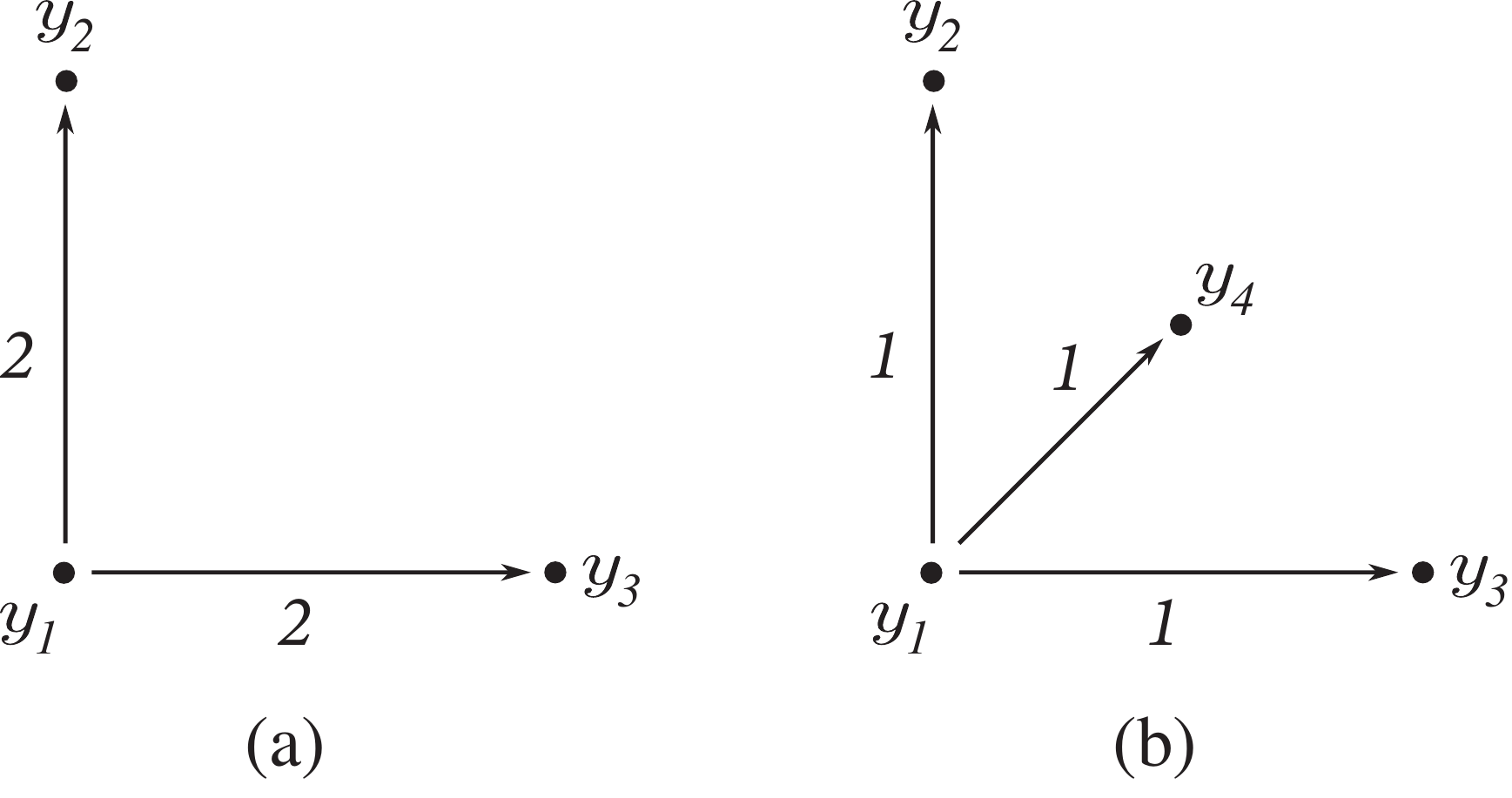}
\caption{\small Examples of dynamically equivalent networks (a) and (b).}
\label{fig:dynamical_equivalence}
\end{figure} 
\end{example}

\section{Weakly reversible deficiency one networks}
\label{sec:weakly_reversible_deficiency_one}

Deficiency analysis \cite{feinberg1987chemical,feinberg1995existence,horn1972necessary,feinberg1995multiple} forms an integral component of reaction network theory. The dynamics generated by reaction networks with low deficiency has been studied extensively using the Deficiency zero and Deficiency one theorems~\cite{feinberg2019foundations,feinberg1980chemical,feinberg1987chemical,feinberg1995existence}. In particular, properties like the existence of a unique equilibrium within each stoichiometric compatibility class, local asymptotic stability of the equilibrium owing to the existence of a Lyapunov function have been established. In this paper, we focus on weakly reversible deficiency one reaction networks. Such networks are ubiquitous in applications, and some noteworthy examples are listed below.

\begin{example}[Edelstein network,
\cite{ferragut2018liouville}]\label{ex:edelstein}
\[
X_1 \rightleftharpoons 2X_1,  \ X_1 + X_2 \rightleftharpoons X_3 \rightleftharpoons X_2
\]
This is a weakly reversible reaction network with deficiency $\delta = 5 - 2- 2 =1$. 
\end{example}

\begin{example}[Symmetry breaking network, \cite{feinberg2019foundations}]
\[
L+2R+P \rightleftharpoons R+Q, \ R+2L+P  \rightleftharpoons 3L+Q, \ P \rightleftharpoons 0  \rightleftharpoons Q
\]
This is a weakly reversible reaction network with the stoichiometric subspace given by:
\begin{equation} \notag
S = \text{span} \{(-1,-1,1,-1)^{\intercal}, \ (1,-1,1,-1)^{\intercal}, \ (0,-1,0,0)^{\intercal}, \ (0,0,-1,0)^{\intercal} \},
\end{equation}
It is a three-dimensional stoichiometric subspace. The network has deficiency $\delta = 7 - 3 - 3 =1$. 
\end{example}

\medskip

\noindent From inequality in Equation~\eqref{delta inequality}, deficiency one networks can be classified into the following types \footnote{Without loss of generality, we always assume $\delta_{1} = \ldots = \delta_{\ell-1} =0, \ \delta_{\ell} = 1$ in Type I networks, and $\delta_{1} = \ldots = \delta_{\ell} =0$ in Type II networks in the rest of this paper.}:

\begin{itemize}
\item $\delta = 1 = \delta_1 + \delta_2 + \cdots + \delta_{\ell}$.
We call this a \defi{Type I} network.

\item $\delta = 1 > \delta_1 + \delta_2 + \cdots + \delta_{\ell}$.
We call this a \defi{Type II} network.
\end{itemize} 

Weakly reversible deficiency one networks for which $\delta = 1 = \delta_1 + \delta_2 + \cdots + \delta_{\ell}$ (Type I) 
fall into the regime of the Deficiency one Theorem, which we state below.

\begin{theorem}[Deficiency One Theorem,
\cite{feinberg1987chemical,feinberg1995existence}]
\label{thm:deficiency_one_theorem}
Consider a reaction network $G$ consisting of $\ell$ linkage classes $L_1,L_2,\cdots, L_{\ell}$. Let us assume that $G$ satisfies the following conditions:
\begin{enumerate}
\item $\delta_i\leq 1$.
\item $\sum\limits_{i=1}^{\ell} \delta_i = \delta$.
\item Each linkage class $L_i$ contains exactly one terminal strongly connected component.
\end{enumerate}
If there exists a $\bk$ for which the mass-action system $(G,\bk)$ possesses a positive equilibrium, then every stoichiometric compatibility class has exactly one positive equilibrium. If $G$ is weakly reversible, then for all values of $\bk$ the mass-action system $(G,\bk)$ possesses a positive equilibrium. 
\end{theorem}

We also state a theorem~\cite{boros2019existence} that guarantees the existence of positive steady states for weakly reversible  systems.

\begin{theorem}[\cite{boros2019existence}]
\label{thm:boros}
For weakly reversible mass-action systems, there exists a positive steady state within each stoichiometric compatibility class. 
\end{theorem}

Using Theorem~\ref{thm:boros} in conjunction with the Deficiency one Theorem, we conclude that for any weakly reversible deficiency one network of Type I, there exists a unique equilibrium within each stoichiometric compatibility class for all values of the rate constants $\bk$.

\medskip

For weakly reversible deficiency one networks of Type II, all linkage classes have deficiency zero and they possess the following geometric property:

\begin{proposition}[{\cite[Theorem 9]{craciun2019realizations}}]
\label{prop_def_zero_affine}
Consider a reaction network $G$. Let $L_1$ be a linkage class of $G$. Then $L_1$ has deficiency zero if and only if its vertices are affinely independent.
\end{proposition}

Recall that a set $X$ is a \textbf{polyhedral cone} if $X = \{\bx: M \bx \leq \textbf{0} \text{ for some matrix } M \}$. Such a cone is convex. It is \textbf{pointed}, or \textbf{strongly convex} if it does not contain a positive dimensional linear subspace. A pointed polyhedral cone admits a unique (up to scalar multiple) minimal set of generators, and these generating vectors are called \emph{extreme vectors}~\cite{2016david}.

\begin{lemma}
\label{lem: pointed cone and generator} 

Consider a mass-action system $(G,\bk)$ with vertices $\{\by_i\}_{i=1}^m$. Let $\bW$ be the matrix of net reaction vectors of $G$, then we have: 
\begin{enumerate}[label=(\alph*)]
\item $\text{ker}(\bW) \cap \RR^m_{\geq 0}$ is a pointed polyhedral cone. 

\item There exists the minimal set of generators for $\text{ker}(\bW) \cap \RR^m_{\geq 0}$.
\end{enumerate}
\end{lemma}

\begin{proof}

\begin{enumerate}[label=(\alph*)]

\item It is clear that the set $\text{ker}(\bW) \cap \RR^m_{\geq 0}$ is the solution to $\bW \nu \geq \textbf{0}, -\bW \nu \geq \textbf{0}$, and $\bI_{m} \nu\geq\textbf{0}$, and the set is a polyhedral cone. From the definition, a cone contained in the positive orthant $\RR^m_{\geq 0}$ is always pointed.
Therefore, we deduce that $\text{ker}(\bW) \cap \RR^m_{\geq 0}$ is a pointed polyhedral cone.

\item Since $\text{ker}(\bW)\cap\mathbb{R}^m_{\geq 0}$ is a pointed cone, by the Minkowski-Weyl theorem~\cite{2016david, rockafellar1970convex}, we have 
\begin{equation}
\text{ker}(\bW)\cap\mathbb{R}^m_{\geq 0} 
= \sum\limits_{j=1}^r \zeta_i \bd_i,
\end{equation}
where $\{\bd_j\}_{j=1}^r $ is the \emph{unique} (up to scalar multiple) minimal set of generators of the cone $\text{ker}(\bW)\cap\mathbb{R}^m_{\geq 0}$. 
\end{enumerate}
\end{proof}

\noindent Furthermore, using the rank-nullity theorem, we have 
\begin{equation} \notag
\text{dim}(\text{ker}(\bW)) = \text{dim}(\text{ker}(\bA_{\bk})) + \text{dim}(\text{ker}(\bY)\cap\text{Im}(\bA_{\bk})).
\end{equation}
Thus the minimal set of generators can be divided into two groups. The next definition illustrates this point.

\begin{definition}[\cite{conradi2007subnetwork}]

Consider a mass-action system $(G,\bk)$ and let $\bW$ be the matrix of net reaction vectors of $G$. An extreme vector $\bd_i$ of the cone $\text{ker}(\bW)\cap\mathbb{R}^m_{\geq 0}$ is called
\begin{enumerate}
\item a \defi{cyclic generator}, if $\bd_i\in\text{ker}(\bA_{\bk})$.
\item a \defi{stoichiometric generator}, if $\bA_{\bk} \bd_i \in \text{ker}(\bY) \backslash 
\{\mathbf{0}\}$.
\end{enumerate}
\end{definition}

Here we give an example where a reaction network possesses both cyclic and stoichiometric generators.

\begin{example}
\label{ex:stoichiometric_gen}

Consider the network shown in Figure \ref{fig:stoichiometric_gen}. This weakly reversible reaction network has two linkage classes, and the deficiency of the entire network is one (i.e. $\delta = 1$). Moreover, the net reaction vector matrix follows:
\begin{equation}
\bW = \begin{pmatrix}
1 & 1   & -2	 & 1 & -1\\
1  & -1 & 0   & 0  & 0 
\end{pmatrix},
\end{equation}
and
\begin{equation}
\text{ker}(\bW)) = \text{span} \big\{ (1,1,1,0,0)^{\intercal}, \ (1,1,0,-2,0)^{\intercal}, \
(1,1,0,0,2)^{\intercal} \big\}.
\end{equation}
Therefore, we can compute the minimal set of generators of $\text{ker}(\bW)\cap\mathbb{R}^5_{\geq 0}$:
\begin{enumerate}
    \item[(i)] Cyclic generators: 
\begin{equation}
\bd_1 = (1,1,1,0,0)^{\intercal}, \ \ 
\bd_2 = (0,0,0,1,1)^{\intercal}.
\end{equation} 

\item[(ii)] Stoichiometric generators:
\begin{equation}
\bd_3 = (1,1,0,0,2)^{\intercal}, \ \ 
\bd_4 = (0,0,1,2,0)^{\intercal}.
\end{equation} 
\end{enumerate}

\begin{figure}[H]
\centering
\includegraphics[scale=0.5]{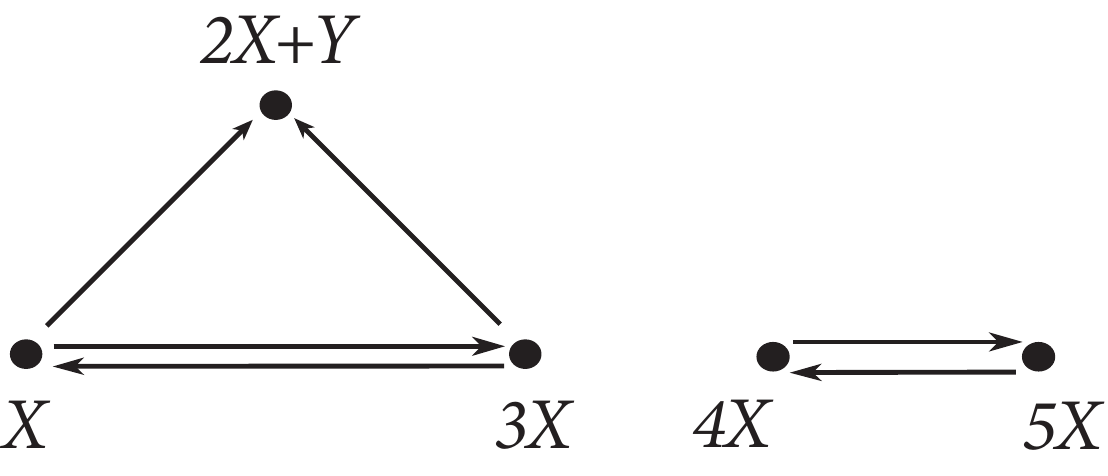}
\caption{The mass-action system corresponds to Example~\ref{ex:stoichiometric_gen}, which has both cyclic and stoichiometric generators.}
\label{fig:stoichiometric_gen}
\end{figure} 
\end{example}

In general, both cyclic and stoichiometric generators can be studied by flux mode analysis.  
Further, Conradi et al. \cite{conradi2007subnetwork} defined subnetworks generated by stoichiometric generators, and showed that under some  conditions if these subnetworks exhibit multistationarity, then so does the original network.

As remarked before, weakly reversible deficiency one realizations of Type I satisfy the conditions of the Deficiency One Theorem. This implies that there exists a unique equilibrium within each stoichiometric compatibility class for all values of the rate constants $\bk$ of these realizations. Weakly reversible deficiency one realizations of Type II are also important since the subnetworks generated by the stoichiometric generators can help answer questions about multistationarity. It is therefore important to identify and analyze weakly reversible deficiency one realizations.

\section{The pointed cone \texorpdfstring{$\text{ker}(\bW)\cap\mathbb{R}^m_{\geq 0}$}{kerwRm>0}}
\label{sec:pointed_cone}

The goal of this section is to analyze the pointed cone $\text{ker}(\bW)\cap\mathbb{R}^m_{\geq 0}$ for weakly reversible deficiency one reaction networks. Specifically, we focus on the extreme vectors of $\text{ker}(\bW)\cap\mathbb{R}^m_{\geq 0}$.

\begin{lemma} \label{lem:wr_ker support}
Consider a weakly reversible mass-action system $(G,\bk)$ with vertices $\{\by_i\}_{i=1}^m$.
Let $\bW$ be the matrix of net reaction vectors of $G$, and $\{\bd_1, \ldots, \bd_r\}$ be the minimal set of generators of $\text{ker}(\bW) \cap \mathbb{R}^m_{\geq 0}$, then 
\begin{equation} \label{lem:supp_whole} 
\bigcup\limits_{i=1}^r \text{supp}(\bd_i) = [m].
\end{equation}
\end{lemma}

\begin{proof}
For contradiction, assume there exists $j \in [m]$, such that $j \notin \bigcup\limits_{i=1}^r \text{supp}(\bd_i)$.  Then for any $\bv = (v_1, \ldots, v_m)^{\intercal} \in\text{ker}(\bW)\cap\mathbb{R}^m_{\geq 0}$, we obtain that
\begin{equation} \label{vj = 0}
\bv_j=0.
\end{equation}
Since $(G,\bk)$ is a weakly reversible mass-action system, by Proposition~\ref{prop:positive_vector_kernel} there exists a positive vector in the kernel of the Kirchoff matrix $\bA_{\bk}$. 
Note that weakly reversibility indicates $\bW = \bY\bA_{\bk}$. Thus we have
\begin{equation} \notag
\text{ker}(\bA_{\bk}) \subseteq \text{ker}(\bY\bA_{\bk}) = \text{ker}(\bW). 
\end{equation}
This implies the existence of a positive vector in $\text{ker}(\bW)\cap\mathbb{R}^m_{\geq 0}$, contradicting Equation~\eqref{vj = 0}.
\end{proof}

\begin{lemma}[\cite{def_one}] \label{lem:wr_ker}
Consider a weakly reversible mass-action system $(G,\bk)$ with vertices $\{\by_i\}_{i=1}^m$ and stoichiometric subspace $S$. Let 
$\bW$ be the matrix of net reaction vectors of $G$, then
\begin{equation} \label{lem:image_stoich} 
\text{Im}(\bW) = S.
\end{equation} 
\end{lemma}

The following lemma concerns the dimension of  $\text{ker}(\bW)$
in various cases.

\begin{lemma} \label{lem:dim wr_ker}
Consider a weakly reversible mass-action system $(G,\bk)$ with vertices $\{\by_i\}_{i=1}^m$ and stoichiometric subspace $S$. Let $\bW$ be the matrix of net reaction vectors of $G$. 
\begin{enumerate}[label=(\alph*)]
\item \label{dim of kerW single linkage class} If $G$ has deficiency $\delta$ and a single linkage class (i.e. $\ell = 1$), we have
\begin{equation}
\text{dim}(\text{ker}(\bW)) = \delta+ 1.
\end{equation}
Moreover, if $\delta = 0$, then for any $\bz\in\text{ker}(\bW) \backslash \{\mathbf{0}\}$, $\text{supp}(\bz)= [m]$.

\item\label{dim of kerW deficiency one} If $G$ has deficiency one and $\ell \geq 1$ linkage classes, we have
\begin{equation}
\text{dim}(\text{ker}(\bW)) = \ell + 1.
\end{equation}
\end{enumerate}
\end{lemma}

\begin{proof}

$(a)$
Since $G$ has deficiency $\delta$ and one linkage class, we have  
\begin{equation} \notag
\dim (S) = s = m- (\delta + 1).
\end{equation} 
By Lemma~\ref{lem:wr_ker}, $\text{rank}(\bW) = \dim(\text{Im}(\bW)) = s$. Using the rank-nullity theorem, we obtain 
\begin{equation} \notag
\text{dim}(\text{ker}(\bW)) = \delta+ 1.
\end{equation}

Furthermore, if $\delta = 0$, we deduce that
\begin{equation} \label{dim ker_w delta 0}
\text{dim}(\text{ker}(\bW)) = 1.
\end{equation}
As a weakly reversible mass-action system, $(G,\bk)$ possesses a strictly positive steady state $\hat{\bx} \in \mathbb{R}^n_{>0}$ by Theorem~\ref{thm:boros}. 
Using Equation~\eqref{eq:mass_action}, we get
\begin{eqnarray} \label{eq:ker_w delta 0}
\sum\limits_{j\in [m]}{\hat{\bx}}^{\by_j}\sum\limits_{\by_j\to \by_j' \in \mathcal{R}} k_{\by_j\to \by_j'}(\by_j'- \by_j) = \sum\limits_{j\in [m]}{\hat{\bx}}^{\by_j} \bw_j = 0. 
\end{eqnarray}
Note that $({\hat{\bx}}^{\by_1}, {\hat{\bx}}^{\by_2},\ldots, {\hat{\bx}}^{\by_m}) \in \mathbb{R}^m_{>0}$, and it spans $\text{ker}(\bW)$ due to Equation~\eqref{dim ker_w delta 0}. 

\medskip

$(b)$
Since the deficiency of $G$ is one, we get 
\begin{equation} \notag
\dim (S) = s = m - \ell - \delta = m - (\ell + 1).
\end{equation}
From Lemma~\ref{lem:wr_ker}, we conclude that 
\begin{equation} \notag
\text{dim}(\text{ker}(\bW)) = m - \dim(\text{Im}(\bW)) = m - \dim (S) = \ell + 1.
\end{equation}
\end{proof}

Here we start with the minimal set of generators of $\text{ker}(\bW) \cap \mathbb{R}^m_{\geq 0}$, when weakly reversible mass-action systems contain a single linkage class.

\begin{lemma} \label{lem:deficiency_one_cases part1}

Consider a weakly reversible mass-action system $(G,\bk)$ that has deficiency $\delta$ and a single linkage class $L = \{ \by_1, \ldots, \by_m \}$. 
Let $\bW$ be the matrix of net reaction vectors of $G$, and $\{\bd_1, \ldots, \bd_r\}$ be the minimal set of generators of $\text{ker}(\bW) \cap \mathbb{R}^m_{\geq 0}$, then 
\begin{equation} \label{eq: generator on deficiency}
r \geq \delta + 1.
\end{equation}
Moreover, if $\delta = 1$, then $r = 2$.
Assume $\{\bd_1,\bd_2\}$ is the minimal set of generators, then 
\begin{equation}
\label{support of deficiency_one_cases part1}
\text{supp}(\bd_1) \subsetneq [m], \ \
\text{supp}(\bd_2) \subsetneq [m], \ \
\text{supp}(\bd_1) \cup \text{supp}(\bd_2) = [m].
\end{equation}
\end{lemma} 

\begin{proof}

Since $G$ has deficiency $\delta$ and one linkage class, from Lemma \ref{lem:dim wr_ker}.\ref{dim of kerW single linkage class} 
we obtain
\begin{equation} \notag
\text{dim}(\text{ker}(\bW)) = \delta +1.
\end{equation}
Using Equation~\eqref{eq:ker_w delta 0} in Lemma \ref{lem:dim wr_ker}, we set $\bd = (\bx^{\by_1}, \bx^{\by_2},\ldots, \bx^{\by_m})$ where $\bx \in \mathbb{R}^n_{>0}$ is a steady state for the system, and obtain $\bd \in \text{ker}(\bW) \cap \mathbb{R}^m_{\geq 0}$.
Then there exists a basis of $\text{ker} (\bW)$ that contains $\bd$ as follows.
\begin{equation} \notag
B = \{ \bd, \be_1, \ldots, \be_{\delta} \}.
\end{equation}
Since $\bd \in \mathbb{R}^m_{>0}$, for any weights
$\lambda_1, \ldots, \lambda_{\delta}$, we can always find a sufficiently large $\lambda > 0$, such that
\begin{equation} \notag
\sum\limits^{\delta}_{i=1} \lambda_i \be_i + \lambda \bd \in  \mathbb{R}^m_{>0}.
\end{equation}
Thus, we conclude 
\begin{equation} \notag
r \geq \text{dim}(\text{ker}(\bW))= \delta +1.
\end{equation}

Furthermore, if the system has deficiency one (i.e. $\delta = 1$), we derive that
\begin{equation} \notag
\text{dim}(\text{ker}(\bW)) = \delta +1 = 2,
\end{equation}
and thus $\text{ker} (\bW) \cap \mathbb{R}^m_{\geq 0}$ is a two-dimensional pointed cone. 
Therefore, the cone $\text{ker} (\bW) \cap \mathbb{R}^m_{\geq 0}$ must have two generators, i.e., $r=2$.

Now assume $\{\bd_1,\bd_2\}$ is the minimal set of generators of $\text{ker} (\bW) \cap \mathbb{R}^m_{\geq 0}$ when the system has deficiency one. Using $\bd = (\bx^{\by_1}, \bx^{\by_2},\ldots, \bx^{\by_m}) \in \mathbb{R}^m_{>0}$, we derive that
\begin{equation} \notag
\text{supp}(\bd_1) \cup \text{supp}(\bd_2) = [m].
\end{equation}
Suppose $\text{supp}(\bd_1) = [m]$, thus $\bd_1 \in \mathbb{R}^m_{>0}$. Then we can find a sufficiently large $\lambda > 0$, such that
\begin{equation} \notag
\bv = \lambda \bd_1 - \bd_2 \in  \text{ker} (\bW) \cap \mathbb{R}^m_{>0}.
\end{equation}
Note that $\bd_1$ and $\bd_2$ are linearly independent, this contradicts with $\{\bd_1,\bd_2\}$ being the generating set of $\text{ker} (\bW) \cap \mathbb{R}^m_{\geq 0}$. Thus, we derive that $\text{supp}(\bd_1) \subsetneq [m]$. Similarly, we can show $\text{supp}(\bd_2) \subsetneq [m]$, and conclude \eqref{support of deficiency_one_cases part1}.
\end{proof}

\begin{lemma}
\label{lem:deficiency_zero generator}

Consider a weakly reversible mass-action system $(G,\bk)$ with a single linkage class $L = \{ \by_1, \ldots, \by_m \}$, and let $\bW$ be the matrix of net reaction vectors of $G$. Then there exists a vector $\bd \in \mathbb{R}^m_{\geq 0}$ generating the cone $\text{ker}(\bW) \cap \mathbb{R}^m_{\geq 0}$ if and only if the system has deficiency zero.
\end{lemma}

\begin{proof}

First, suppose the system has deficiency zero.
From Equation~\eqref{eq:ker_w delta 0} in Lemma \ref{lem:dim wr_ker}, we set $\bd = (\bx^{\by_1}, \bx^{\by_2},\ldots, \bx^{\by_m})$ where $\bx \in \mathbb{R}^n_{>0}$ is a steady state for the system $(G,\bk)$, and obtain 
\begin{equation} \notag
\text{ker}(\bW) = \text{span}\{ \bd \}.
\end{equation}
One can check $\bd \in \mathbb{R}^m_{>0}$, and hence $\bd$ generates $\text{ker}(\bW) \cap \mathbb{R}^m_{\geq 0}$. 

On the other hand, consider a vector $\bd$ that generates $\text{ker}(\bW) \cap \mathbb{R}^m_{\geq 0}$. Using~\eqref{eq: generator on deficiency} and deficiency is non-negative, we get that
\begin{equation}
0 \leq \delta \leq 1 - 1 = 0.
\end{equation}
Thus, we conclude the deficiency of the system is zero.
\end{proof}

The remark below follows from Lemmas~\ref{lem:deficiency_one_cases part1} and~\ref{lem:deficiency_zero generator}.

\begin{remark}
\label{rmk:generator_on_deficiency_one}

Consider a weakly reversible mass-action system $(G,\bk)$ with a single linkage class $L = \{\by_i\}_{i=1}^m$. 
Let $\bW$ be the matrix of net reaction vectors of $G$.
Suppose two vectors $\bd_1, \bd_2$ form the minimal set of generators of $\text{ker}(\bW) \cap \mathbb{R}^m_{\geq 0}$, then $(G,\bk)$ has deficiency one.
\end{remark}

Next, we work on the minimal set of generators of $\text{ker}(\bW)\cap\mathbb{R}^m_{\geq 0}$ when the weakly reversible deficiency one networks have multiple linkage classes.

\begin{lemma} \label{lem:deficiency_one_cases part2}
Consider a weakly reversible deficiency one mass-action system $(G,\bk)$ of Type I that has $\ell > 1$ linkage classes, denoted by $L_1, \ldots, L_{\ell}$. Let $\bW$ be the matrix of net reaction vectors of $G$, and $\{\bW_p \}^{\ell}_{p=1}$ be the matrix of net vectors corresponding to linkage classes $\{L_p \}^{\ell}_{p=1}$, then

\begin{enumerate}[label=(\alph*)]

\item  
\begin{equation} \label{dim of deficiency_one_cases part2}
\text{dim}(\text{ker}(\bW_1)) + \cdots + \text{dim}(\text{ker}(\bW_{\ell}))  =  \text{dim}(\text{ker}(\bW)) = \ell + 1,
\end{equation}
where
\begin{equation} \notag
\text{dim}(\text{ker}(\bW_i)) =
\begin{cases}
    1, & \text{for } 1 \leq i \leq \ell-1, \\[5pt]
    2, & \text{for } i = \ell.
\end{cases}
\end{equation}
Moreover, for any $1 \leq i \leq \ell-1$ and $\bz \in \text{ker}(\bW_i) \backslash \{\mathbf{0}\}$, we have $\text{supp} (\bz) = L_i$.

\item There exist $\ell + 1$ vectors $\bd_1, \ldots, \bd_{\ell + 1}$, which form the minimal set of generators of the cone $\text{ker} (\bW)\cap\mathbb{R}^m_{\geq 0}$, such that 
\begin{equation}
\label{support of deficiency_one_cases part2 a}
\text{supp}(\bd_i) = L_i, \ \text{for } 1 \leq i \leq \ell-1,
\end{equation}
and
\begin{equation} 
\label{support of deficiency_one_cases part2 b}
\text{supp}(\bd_{\ell}) \subsetneq L_{\ell}, \ \
\text{supp}(\bd_{\ell+1}) \subsetneq L_{\ell}, \ \ 
\text{supp}(\bd_{\ell})\cup \text{supp}(\bd_{\ell +1}) = L_{\ell}.
\end{equation}
\end{enumerate}
\end{lemma}

\begin{proof}


$(a)$ From the assumption, the system $G$ is of Type I with $\delta_{1} = \cdots = \delta_{\ell-1} =0$ and $\delta_{\ell} = 1$.
Using Lemma~\ref{lem:dim wr_ker}, we get 
\begin{equation} \notag
\text{dim}(\text{ker}(\bW_{\ell})) = \delta_{\ell} + 1 = 2, 
\ \
\text{dim}(\text{ker}(\bW_i)) = \delta_{i} + 1 = 1, \ \text{for } 1 \leq i \leq \ell-1.
\end{equation}
Further, for any $1 \leq i \leq \ell-1$ and $\bz \in \text{ker}(\bW_i) \backslash \{\mathbf{0}\}$,
\begin{equation} \notag
\text{supp}(\bz) = L_i.
\end{equation}
Note that $G$ has deficiency one, thus
$\text{dim}(\text{ker}(\bW)) = \ell + 1$,
and we derive \eqref{dim of deficiency_one_cases part2}.

\medskip


$(b)$ 
Now we construct the minimal set of generators of $\text{ker}(\bW) \cap \mathbb{R}^m_{\geq 0}$, denoted by $\{\bd_1, \ldots,\bd_{r}\}$. It will follow from the construction that $r = \ell + 1=  \text{dim}(\text{ker}(\bW))$. 

Since $(G,\bk)$ is a weakly reversible mass-action system, it possesses a strictly positive steady state $\hat{\bx}\in \mathbb{R}^n_{>0}$ by Theorem~\ref{thm:boros}. Following Equation~\eqref{eq:ker_w delta 0} in Lemma~\ref{lem:dim wr_ker}, we can build $\ell-1$ vectors $\bd_1, \ldots, \bd_{\ell - 1}$. We define $\bd_1 = (\bd_{1,1},\ldots,\bd_{1,m}) \in \text{ker}(\bW) \cap \mathbb{R}^m_{\geq 0}$, such that
\begin{eqnarray} \label{eq:defn_generators_1 a}
\bd_{1,i} =  
\begin{cases}
    \hat{\bx}^{\by_i}, & \text{for } i \in L_1, \\
    0, & \text{for } i \notin L_1. 
\end{cases}
\end{eqnarray} 
It is clear that $\text{supp}(\bd_1) = L_1$. Analogously, for $i=1, \ldots, \ell - 1$, we can define $\bd_i$ corresponding to the linkage classes $L_i$ with $\text{supp}(\bd_i) = L_i$. 

Note that $G$ is of Type I and the linkage class $L_{\ell}$ has deficiency one. 
Let $L_{\ell} = \big\{ \by_{\ell_i} \big\}^{m_{\ell}}_{i = 1}$ with $m_{\ell} = |L_{\ell}|$. From Lemma~\ref{lem:deficiency_one_cases part1}, the cone $\text{dim}(\text{ker}(\bW_{\ell}))\cap\mathbb{R}^{m_{\ell}}_{\geq 0}$ has two generators, denoted by $\{ \bg_1, \bg_2 \}$. Suppose $\bg_1 = \big( \bg_{1,i} \big)_{i \in L_{\ell}}$ and   $\bg_2 = \big( \bg_{2,i} \big)_{i \in L_{\ell}}$, then we define $\bd_{\ell} = (\bd_{\ell, 1}, \ldots,\bd_{\ell, m})$, $\bd_{\ell+1} = (\bd_{\ell+1,1}, \ldots, \bd_{\ell+1,m})$ as
\begin{equation} 
\label{eq:defn_generators_1 b}
\begin{split}
\bd_{\ell,i} 
 = \begin{cases}
    \bg_{1,i}, & \text{for } i \in L_{\ell}, \\
    0, & \text{for } i \notin L_{\ell}.
    \end{cases} 
\ \ \text{and } \
\bd_{\ell + 1,i} 
 = \begin{cases}
    \bg_{2,i}, & \text{for } i \in L_{\ell}, \\
    0, & \text{for } i \notin L_{\ell}.
    \end{cases}
\end{split}
\end{equation}
Note that both $\bd_{\ell}, \bd_{\ell+1} \in \text{ker}(\bW) \cap \mathbb{R}^m_{\geq 0}$, and satisfy Equation~\eqref{support of deficiency_one_cases part2 b}.

We claim that the vectors $\bd_1, \ldots , \bd_{\ell}$ form a set of generators for $\text{ker}(\bW) \cap \mathbb{R}^m_{\geq 0}$. 
From Equations~\eqref{eq:defn_generators_1 a} and~\eqref{eq:defn_generators_1 b}, we deduce that the vectors $\{ \bd_i \}^{\ell + 1}_{i=1}$ are linearly independent. Together with $\text{dim}(\text{ker}(\bW))=\ell + 1$, we derive that the set $\{\bd_1, \ldots ,\bd_{\ell + 1}\}$ is a basis for $\text{ker}(\bW)$. Thus, any vector $\bv \in \text{ker}(\bW) \cap \mathbb{R}^m_{\geq 0}$ can be expressed as
\begin{equation} \label{linear combination 1}
\bv = a_1 \bd_1 + a_2 \bd_2 + \cdots + a_{\ell+1} \bd_{\ell+1} \in \mathbb{R}^m_{\geq 0},
\end{equation}
where $a_1, \ldots, a_{\ell+1} \in \mathbb{R}$. 
So it suffices to prove all $\{ a_i \}^{\ell+1}_{i=1}$ are non-negative. 
Recall $\{ L_i\}^{\ell}_{i=1}$ are linkage classes with $\text{supp}(\bd_i) = L_i$, and $\text{supp}(\bd_{\ell})$, $\text{supp}(\bd_{\ell+1}) \subseteq L_{\ell}$, then we obtain 
\begin{equation} \notag
a_i \geq 0, \ \text{for } i=1, \ldots, \ell-1.
\end{equation}
Moreover, we set $\hat{\bv} = a_{\ell} \bg_1 + a_{\ell+1} \bg_2$. From $\sum\limits^{\ell}_{i=1} \text{dim}(\text{ker}(\bW_i)) = \text{dim}(\text{ker}(\bW))$, we derive
\begin{equation} \notag
\hat{\bv} \in \text{ker}(\bW_{\ell})\cap\mathbb{R}^{m_{\ell}}_{\geq 0}.
\end{equation}
Since $\bg_1, \bg_2$ form the generators of the cone $\text{dim}(\text{ker}(\bW_{\ell}))\cap\mathbb{R}^{m_{\ell}}_{\geq 0}$, we have
\begin{equation} \notag
a_{\ell} \geq 0, \ \ 
a_{\ell+1} \geq 0.
\end{equation}
Therefore, we prove the claim.  

Finally, we show $\{\bd_1, \bd_2, \ldots, \bd_{\ell +1}\}$ is the minimal set of generators for $\text{ker}(\bW) \cap \mathbb{R}^m_{\geq 0}$. 
Note from Equations~\eqref{eq:defn_generators_1 a} and~\eqref{eq:defn_generators_1 b}, $\bd_{\ell + 1}$ cannot be generated by $\{\bd_i\}^{\ell}_{i=1}$, thus it suffices to show $\{\bd_1,\bd_2, \ldots ,\bd_{\ell}\}$ are all extreme vectors. 

Suppose not, there exists $1 \leq j \leq \ell$, such that $\bd_j$ is not an extreme vector.
Then we can find two vectors $\gamma, \theta \in \text{ker}(\bW) \cap \mathbb{R}^m_{\geq 0}$ and $0 < \lambda < 1$, such that
\begin{equation} \label{d_decompose_1}
\lambda \gamma + (1 - \lambda) \theta = \bd_j,
\end{equation}
where $\gamma \neq \nu \theta$ for any constant $\nu$.
From Equation~\eqref{linear combination 1}, we can express $\gamma$ and $\theta$ as the conical combination of $\{\bd_1,\bd_2, \ldots ,\bd_{\ell + 1}\}$ as
\begin{equation} \notag
\gamma = \sum\limits^{\ell+1}_{i=1} \gamma_{i}\bd_i
\ \ \text{and } \
\theta = \sum\limits^{\ell+1}_{i=1} \theta_{i} \bd_i,
\end{equation}
where  $\gamma_{i}, \ \theta_{i} \geq 0$, for $i = 1, \ldots, \ell+1$. 

If $j \neq \ell$, from $\text{supp}(\bd_j) = L_j$ and Equation~\eqref{d_decompose_1}, we derive that $\gamma_{i} = \theta_{i} = 0$ for $1 \leq i \leq \ell+1, i \neq j$. 
This implies $\gamma = \gamma_j d_j$ and $\theta = \theta_j d_j$, which contradicts with $\gamma \neq \nu \theta$. 

If $j = \ell$, we deduce that $\gamma_{i} = \theta_{i} = 0$ for $1 \leq i \leq \ell+1$ such that $i \neq j$ in a similar way. This implies $\gamma = \gamma_{\ell}\bd_{\ell}$ and $\theta = \theta_{\ell}\bd_{\ell}$, which also contradicts with $\gamma \neq \nu \theta$. 
Therefore we conclude that $\{\bd_1, \ldots, \bd_{\ell + 1}\}$ is the minimal set of generators of the cone $\text{ker}(\bW) \cap \mathbb{R}^m_{\geq 0}$.
\end{proof}

Here we illustrate an example where Lemma \ref{lem:deficiency_one_cases part2} can be verified.

\begin{example}
\label{ex:deficiency_sum_lemma}

Consider a weakly reversible deficiency one mass-action system 
shown in Figure \ref{fig:deficiency_sum_lemma}.
This reaction network has two linkage classes.  One linkage class has deficiency zero, and the other has deficiency one (i.e. $\delta_1 = 0, \ \delta_2 = 1$), and the deficiency of the entire network is one (i.e. $\delta = 1$). Therefore, we have 
\begin{equation}
    1 = \delta = \delta_1 + \delta_2.
\end{equation}
For all reactions $\by\rightarrow\by' \in E$, we assume
$k_{\by\rightarrow\by'} = 1$, and get
\begin{equation} 
\bW_1 = 
\begin{pmatrix}
1 & -1\\
0  & 0
\end{pmatrix}, \ \ 
\bW_2 = 
\begin{pmatrix}
0 &  0 & 0\\
1  &  0  &-1
\end{pmatrix}, \ \ 
\bW = 
\begin{pmatrix}
1 & -1 & 0 &  0 & 0 \\
0  & 0 & 1  &  0  & -1
\end{pmatrix}.
\end{equation}
So we can derive that
\begin{equation}
\text{ker}(\bW_1)) = \text{span} 
\left\{
\begin{pmatrix}
1 \\
1 \\
\end{pmatrix}
\right\}, \
\text{ker}(\bW_2)) = \text{span} 
\left\{ \begin{pmatrix}
0 \\
1 \\
0 \\
\end{pmatrix}, \ \ 
\begin{pmatrix}
1 \\
0 \\
1 \\
\end{pmatrix}\right\},
\end{equation} 
and
\begin{equation}
\text{ker}(\bW)) = \text{span} 
\left\{\begin{pmatrix}
1 \\
1 \\
0 \\
0 \\
0 \\
\end{pmatrix}, \ \ 
\begin{pmatrix}
0 \\
0 \\
0 \\
1 \\
0 \\
\end{pmatrix}, \ \ 
\begin{pmatrix}
0 \\
0 \\
1 \\
0 \\
1 \\
\end{pmatrix}\right\}.
\end{equation}
For any vector $\bz_1 \in \text{ker}(\bW_1) \backslash \{\mathbf{0}\}$, we have 
\begin{equation}
\text{supp}(\bz_1)= L_1. 
\end{equation}
Then, we compute the minimal set of generators of $\text{ker}(\bW)\cap\mathbb{R}^5_{\geq 0}$:
\begin{equation}
\bd_2 = 
\begin{pmatrix}
1 \\
1 \\
0 \\
0 \\
0 \\
\end{pmatrix}, \ \ 
\bd_2 = 
\begin{pmatrix}
0 \\
0 \\
0 \\
1 \\
0 \\
\end{pmatrix}, \ \ 
\bd_3 = \begin{pmatrix}
0 \\
0 \\
1 \\
0 \\
1 \\
\end{pmatrix}
\end{equation}
This shows that the number of extreme vectors: $r = 3$, and
\begin{equation}
r = \text{dim}(\text{ker}(\bW)) = \ell +1,
\end{equation}
where $\ell = 2$.
Moreover, for $q = 2, 3$,
\begin{equation}
\text{supp}(\bd_1) = L_1, \ \ 
\text{supp}(\bd_{q}) \cap L_2 \subsetneq L_2, \ \
\text{supp}(\bd_{2})\cup \text{supp}(\bd_{3}) = L_{2}.
\end{equation}
Therefore, we verify Lemma~\ref{lem:deficiency_one_cases part2}.

\begin{figure}[H]
\centering
\includegraphics[scale=0.5]{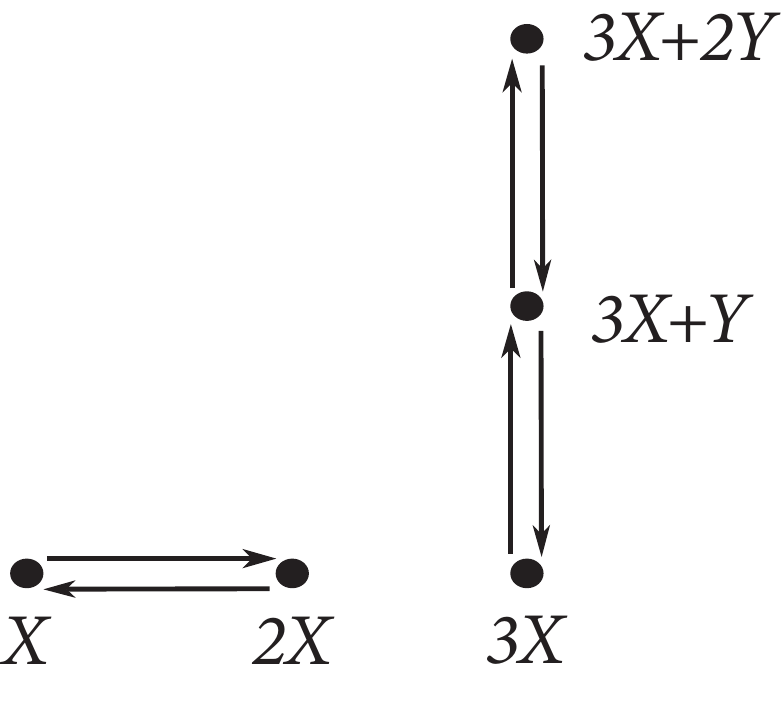}
\caption{A weakly reversible deficiency one mass-action system of Type I from Example \ref{ex:deficiency_sum_lemma}}
\label{fig:deficiency_sum_lemma}
\end{figure} 
\end{example}

\begin{lemma} \label{lem:deficiency_one_cases part3}

Consider a weakly reversible deficiency one mass-action system $(G,\bk)$ of Type II that has $\ell > 1$ linkage classes denoted by $L_1, \ldots, L_{\ell}$. Let $\bW$ be the matrix of net reaction vectors of $G$, and $\{\bW_p \}^{\ell}_{p=1}$ be the matrix of net reaction vectors corresponding to linkage classes $\{L_p \}^{\ell}_{p=1}$, then

\begin{enumerate}[label=(\alph*)]
\item[(a)]
\begin{equation} \label{dim of deficiency_one_cases part3}
\text{dim}(\text{ker}(\bW_1)) + \cdots + \text{dim}(\text{ker}(\bW_{\ell}))  =  \text{dim}(\text{ker}(\bW)) - 1 = \ell,
\end{equation}
where
\begin{equation} \notag
\text{dim}(\text{ker}(\bW_i)) = 1, \ \text{for } 1 \leq i \leq \ell.
\end{equation}
Moreover, for any $1 \leq i \leq \ell$ and $\bz \in \text{ker}(\bW_i) \backslash \{\mathbf{0}\}$, we have $\text{supp} (\bz) = L_i$.

\item[(b)] There exist $\ell + 2$ vectors $\bd_1, \ldots, \bd_{\ell + 2}$, which form the minimal set of generators of the cone $\text{ker} (\bW)\cap\mathbb{R}^m_{\geq 0}$, such that for $i = 1, \ldots, \ell$,
\begin{equation}
\label{support of deficiency_one_cases part 3}
\text{supp}(\bd_i) = L_i, \ \
\emptyset \neq \text{supp}(\bd_{\ell+1}) \cap L_i \subsetneq L_{i}, \ \
\emptyset \neq \text{supp}(\bd_{\ell+2}) \cap L_i \subsetneq L_{i}.
\end{equation}
\end{enumerate}

\end{lemma}

\begin{proof}



$(a)$
From the assumption, the system $G$ is of Type II with $\delta_{1} = \cdots = \delta_{\ell} =0$.
Using Lemma~\ref{lem:dim wr_ker}, we get for $i=1, \ldots, \ell$, 
\begin{equation} \notag
\text{dim}(\text{ker}(\bW_i)) = \delta_{i} + 1 = 1.
\end{equation}

Further, for any $1 \leq i \leq \ell$ and $\bz \in \text{ker}(\bW_i) \backslash \{\mathbf{0}\}$,
\begin{equation} \notag
\text{supp}(\bz) = L_i.
\end{equation}
Note that $G$ has deficiency one, thus
$\text{dim}(\text{ker}(\bW)) = \ell + 1$,
and we derive \eqref{dim of deficiency_one_cases part3}.

\medskip


$(b)$ Now we construct the minimal set of generators of $\text{ker}(\bW) \cap \mathbb{R}^m_{\geq 0}$, denoted by $\{\bd_1, \ldots,\bd_{r}\}$. It will follow from the construction that $r = \ell + 2 = \text{dim}(\text{ker}(\bW)) + 1$.

Since $(G,\bk)$ is a weakly reversible mass-action system, it possesses a strictly positive steady state $\hat{\bx}\in \mathbb{R}^n_{>0}$ by Theorem~\ref{thm:boros}. Following Equation~\eqref{eq:ker_w delta 0} in Lemma~\ref{lem:dim wr_ker}, we can build $\ell$ vectors $\bd_1, \ldots, \bd_{\ell}$. We define $\bd_1 = (\bd_{1,1},\ldots,\bd_{1,m}) \in \text{ker}(\bW) \cap \mathbb{R}^m_{\geq 0}$, such that
\begin{eqnarray} \label{eq:defn_generators_2 a}
\bd_{1,i} =  
\begin{cases}
    \bx^{\by_i}, & \text{for } i \in L_1, \\
    0, & \text{for } i \notin L_1, 
\end{cases}
\end{eqnarray} 
It is clear that $\text{supp}(\bd_1) = L_1$. Analogously, for $i=1, \ldots, \ell$, we can define $\bd_i$ corresponding to the linkage classes $L_i$, with $\text{supp}(\bd_i) = L_i$.

Now we show that there exists a non-zero vector $\bd_{\ell + 1} \in \text{ker}(\bW) \cap \mathbb{R}^m_{\geq 0}$, such that 
\begin{equation} \label{support l+1}
\text{supp}(\bd_{\ell + 1}) \cap L_i \subsetneq L_i, \ \text{for } i = 1, \ldots, \ell.
\end{equation}
From Equation~\eqref{dim of deficiency_one_cases part3}, there exists a vector $\tilde{\bd} \in \text{ker}(\bW) \backslash \{\mathbf{0}\}$, which is linearly independent from $\{\bd_i\}^{\ell}_{i=1}$. 
Since $\bd_i \in \text{ker}(\bW) \cap \mathbb{R}^m_{\geq 0}$ with $\text{supp}(\bd_i) = L_i$, we set for $i = 1, \ldots, \ell$, 
\begin{equation} \label{max weight l+1}
\alpha_i = \max_{k \in L_i} \Big\{ - \frac{\tilde{\bd}_k}{\bd_{i,k}} \Big\}.
\end{equation} 
Then we define $\bd_{\ell + 1}$ as
\begin{equation} \label{eq:defn_generators_2 b}
\bd_{\ell + 1} 
= \sum\limits^{\ell}_{i=1} \alpha_i \bd_i + \tilde{\bd}.
\end{equation}
For any $1 \leq j \leq  \ell$ and $\theta \in L_j$, we obtain that
\begin{equation} \notag
\bd_{\ell + 1, \theta} = \alpha_j \bd_{i, \theta} + \tilde{\bd}_{\theta} 
\geq  - \frac{\tilde{\bd}_\theta}{\bd_{i,\theta}} \bd_{i, \theta} + \tilde{\bd}_{\theta} 
= 0,
\end{equation}
and the inequality holds when $\theta = k \in L_j$ in Equation~\eqref{max weight l+1}. 
Moreover, the linear independence between $\tilde{\bd}$ and $\{\bd_i\}^{\ell}_{i=1}$ implies that $\bd_{\ell + 1}$ is non-zero.
Thus, we show $\bd_{\ell+1} \in \text{ker}(\bW) \cap \mathbb{R}^m_{\geq 0}$, and it satisfies Equation~\eqref{support l+1}.

Furthermore, we claim that there exist at least two linkage classes: $L_i$, $L_j$ with $1 \leq i, j \leq \ell$ and $i \neq j$, such that
\begin{equation} \label{support l+1 on two classes}
\text{supp}(\bd_{\ell + 1}) \cap L_i \neq \emptyset, \ \text{supp}(\bd_{\ell + 1}) \cap L_j \neq \emptyset.
\end{equation}
Suppose not, we assume that only the linkage class $L_1$ satisfies $\text{supp}(\bd_{\ell + 1}) \cap L_1 \neq \emptyset$. This implies that
\begin{equation} \notag
\text{supp}(\bd_{\ell + 1}) \cap L_1 \subsetneq L_1.
\end{equation}
Using $\text{dim}(\text{ker}(\bW_1))=1$, we get that  $\bd_{\ell + 1}$ must be a scalar multiple of $\bd_{1}$, contradicting Equation~\eqref{eq:defn_generators_2 b}.

\medskip

Next, we construct another non-zero vector $\bd_{\ell + 2} \in \text{ker}(\bW) \cap \mathbb{R}^m_{\geq 0}$, such that
\begin{equation} \label{support l+2}
\text{supp}(\bd_{\ell + 2}) \cap L_i \subsetneq L_i, \ \text{for} \ i = 1, \ldots, \ell.
\end{equation}
Given $\bd_{1}, \ldots, \bd_{\ell}, \bd_{\ell + 1} \in \mathbb{R}^m_{\geq 0}$, we set for $i = 1, \ldots, \ell$, 
\begin{equation} \label{max weight l+2}
\beta_i = \max_{k \in L_i}\Big\{ \frac{\bd_{\ell+1,k}}{\bd_{i,k}} \Big\}.
\end{equation}
It is clear that $\beta_i\geq 0$ for $1\leq i\leq \ell$, then we define $\bd_{\ell + 2}$ as
\begin{equation} \label{eq:defn_generators_2 c}
\bd_{\ell + 2}= \sum\limits^{\ell}_{i=1} \beta_i \bd_i - \bd_{\ell + 1}.
\end{equation}
For any $1 \leq j \leq  \ell$ and $\theta \in L_j$, we get
\begin{equation} \notag
\bd_{\ell + 2, \theta} = \beta_j \bd_{i, \theta} - \bd_{\ell + 1, \theta} 
\geq \frac{\bd_{\ell+1,\theta}}{\bd_{i,\theta}} \bd_{i, \theta} - \bd_{\ell + 1, \theta} 
= 0
\end{equation}
The inequality holds when $\theta = k \in L_j$ in Equation \eqref{max weight l+2}. Moreover, the linear independence between 
$\bd_{\ell+1}$ and $\{\bd_i\}^{\ell}_{i=1}$ implies that $\bd_{\ell + 2}$ is non-zero. Thus, we show $\bd_{\ell+2} \in \text{ker}(\bW) \cap \mathbb{R}^m_{\geq 0}$, and it satisfies Equation~\eqref{support l+2}. Similarly as in \eqref{support l+1 on two classes}, there also exist at least two linkage classes: $L_i$, $L_j$ with $1 \leq i, j \leq \ell$ and $i \neq j$, such that
\begin{equation} \label{support l+2 on two classes}
\text{supp}(\bd_{\ell + 2}) \cap L_i \neq \emptyset, \ \ 
\text{supp}(\bd_{\ell + 2}) \cap L_j \neq \emptyset.
\end{equation}

We claim that the vectors $\bd_1, \ldots , \bd_{\ell+2}$ form a set of generators of $\text{ker}(\bW) \cap \mathbb{R}^m_{\geq 0}$. Using Equations~\eqref{eq:defn_generators_2 a} and~\eqref{eq:defn_generators_2 b}, we deduce the vectors $\{ \bd_i \}^{\ell + 1}_{i=1}$ are linearly independent. Together with $\text{dim}(\text{ker}(\bW))=\ell + 1$, we get that the set $\{\bd_1, \ldots ,\bd_{\ell + 1}\}$ is a basis for $\text{ker}(\bW)$. Thus, any vector $\bv \in \text{ker}(\bW) \cap \mathbb{R}^m_{\geq 0}$ can be expressed as
\begin{equation} \label{linear combination 2a}
\bv = a_1 \bd_1 + a_2 \bd_2 + \cdots + a_{\ell+1} \bd_{\ell+1} \in \mathbb{R}^m_{\geq 0},
\end{equation}
where $a_1, \ldots, a_{\ell+1} \in \mathbb{R}$.
Recall $\{ L_i\}^{\ell}_{i=1}$ are linkage classes with $\text{supp}(\bd_i) = L_i$, for $i=1, \ldots, \ell$, and $\text{supp}(\bd_{\ell + 1})$ in Equation~\eqref{support l+1}, then we obtain
\begin{equation} \notag
a_i \geq 0, \ \text{for } i=1, \ldots, \ell.
\end{equation}
If $a_{\ell+1} \geq 0$, it is clear that $\bv$ can be expressed as a conical combination of $\{\bd_1,\bd_2, \ldots ,\bd_{\ell + 1}\}$ from Equation~\eqref{linear combination 2a}.  Otherwise, if $a_{\ell+1} < 0$, we rewrite $\bv$ as 
\begin{equation} \label{linear combination 2b}
\begin{split}
\bv 
& = a_1 \bd_1 + \cdots + a_{\ell} \bd_{\ell} + a_{\ell+1} ( \sum\limits^{\ell}_{i=1} \beta_i \bd_i - \bd_{\ell +2} )
\\& = a_1\bd_1 + \cdots  + a_{\ell} \bd_{\ell} + a_{\ell+1} \sum\limits^{\ell}_{i=1} \beta_i\bd_i - a_{\ell+1} \bd_{\ell +2}
\\& = (a_1 + a_{\ell+1} \beta_1) \bd_1 + \cdots  + (a_{\ell} + a_{\ell+1} \beta_{\ell}) \bd_{\ell} - a_{\ell+1} \bd_{\ell +2}.
\end{split}
\end{equation}
Using $\bv \in \text{ker}(\bW) \cap \mathbb{R}^m_{\geq 0}$ and Equation~\eqref{support l+2}, we get that for $i = 1, \ldots, \ell$,
\begin{equation} \notag
a_i + a_{\ell+1} \beta_i \geq 0,
\end{equation}
which implies that $\bv$ can be generated by $\{\bd_1, \ldots, \bd_{\ell}, \bd_{\ell+2}\}$.

Finally, we show $\{\bd_1,\bd_2, \ldots ,\bd_{\ell + 2}\}$ is the minimal set of generators for $\text{ker}(\bW) \cap \mathbb{R}^m_{\geq 0}$.  Note that $\bd_1, \ldots ,\bd_{\ell + 1}$ form a basis for $\text{ker}(\bW)$ and $\bd_{\ell + 2} = \sum\limits^{\ell}_{i=1} \beta_i \bd_i - \bd_{\ell + 1}$, thus $\bd_{\ell + 2}$ cannot be generated by $\{\bd_i\}^{\ell+1}_{i=1}$. 
So it suffices to show $\{\bd_1,\bd_2, \ldots ,\bd_{\ell + 1}\}$ are all extreme vectors. 

Suppose not, there exists $1 \leq j \leq \ell+1$, such that $\bd_j$ is not an extreme vector.
Then we can find two vectors $\gamma, \theta \in \text{ker}(\bW) \cap \mathbb{R}^m_{\geq 0}$ and $0 < \lambda < 1$, such that
\begin{equation} \label{d_decompose_2}
\lambda \gamma + (1 - \lambda) \theta = \bd_j,
\end{equation}
where $\gamma \neq \nu \theta$ for any constant $\nu$. Then we write $\gamma$ and $\theta$ as the combination of 
$\{\bd_i\}^{\ell + 1}_{i=1}$,
\begin{equation} \notag
\gamma = \sum\limits^{\ell+1}_{i=1} \gamma_{i}\bd_i, \ \ 
\theta = \sum\limits^{\ell+1}_{i=1} \theta_{i} \bd_i.
\end{equation}
Since $\gamma, \theta \in \mathbb{R}^m_{\geq 0}$,
we have for $i = 1, \ldots, \ell$,
\begin{equation} \notag
\gamma_{i} \geq 0, \ \theta_{i} \geq 0.
\end{equation}

If $j \neq \ell + 1$, from Equation~\eqref{d_decompose_2}, we can derive that $\gamma_{i} = \theta_{i} = 0$ when $1 \leq i \leq \ell, i \neq j$. Since $\gamma, \theta \in \mathbb{R}^m_{\geq 0}$, and $\text{supp}(\bd_{\ell +1}) \cap L_i \subsetneq L_i$ for $i=1, \ldots, \ell$, we derive $\gamma_{\ell + 1} = \theta_{\ell + 1} = 0$. 
This implies $\gamma = \gamma_j d_j$ and $\theta = \theta_j d_j$, and this contradicts with $\gamma \neq \nu \theta$. 

If $j = \ell + 1$, in a similar way we can deduce that $\gamma_{i} = \theta_{i} = 0$ for $1 \leq i \leq \ell$. This implies $\gamma = \gamma_{\ell+1}\bd_{\ell+1}$ and $\theta = \theta_{\ell+1}\bd_{\ell+1}$, which also contradicts with $\gamma \neq \nu \theta$. 
Therefore we conclude that $\{\bd_1, \ldots, \bd_{\ell + 2}\}$ is the minimal set of generators of the cone $\text{ker}(\bW) \cap \mathbb{R}^m_{\geq 0}$.  
\end{proof}

We also illustrate an example where Lemma \ref{lem:deficiency_one_cases part3} can be verified.

\begin{example}
\label{ex:deficiency_not_sum_lemma}

Consider a weakly reversible deficiency one mass-action system 
shown in Figure \ref{fig:deficiency_not_sum_lemma}.
This reaction network has two deficiency zero linkage classes (i.e. $\delta_1 = \delta_2 = 0$), and the deficiency of the entire network is one (i.e. $\delta = 1$).
Therefore, we have 
\begin{equation}
    1 = \delta > \delta_1 + \delta_2 = 0 + 0.
\end{equation}
For all reactions $\by\rightarrow\by' \in E$, we assume
$k_{\by\rightarrow\by'} = 1$, and get
\begin{equation} 
\bW_1 = \begin{pmatrix}
1 & -1\\
0  & 0
\end{pmatrix}, \ \ 
\bW_2 = \begin{pmatrix}
1 & -1\\
0  & 0
\end{pmatrix}, \ \ 
\bW = \begin{pmatrix}
1 & -1 & 1 &  -1 \\
0  & 0 & 0  &  0  
\end{pmatrix}.
\end{equation}
So we can derive that
\begin{equation}
\text{ker}(\bW_1)) = \text{ker}(\bW_2)) = \text{span} 
\left\{
\begin{pmatrix}
1 \\
1 \\
\end{pmatrix}
\right\},
\end{equation}
and
\begin{equation}
\text{ker}(\bW)) = \text{span} 
\left\{\begin{pmatrix}
1 \\
1 \\
0 \\
0 \\
\end{pmatrix}, \ \ 
\begin{pmatrix}
-1 \\
0 \\
1 \\
0 \\
\end{pmatrix}, \ \ 
\begin{pmatrix}
1 \\
0 \\
0 \\
1 \\
\end{pmatrix}\right\}.
\end{equation}
For any vectors $\bz_1 \in \text{ker}(\bW_1) \backslash \{\mathbf{0}\}$ and $\bz_2 \in \text{ker}(\bW_2) \backslash \{\mathbf{0}\}$, we have 
\begin{equation}
\text{supp}(\bz_1)= L_1, \ \ 
\text{supp}(\bz_2)= L_2.
\end{equation}
Then, we compute the minimal set of generators of $\text{ker}(\bW)\cap\mathbb{R}^4_{\geq 0}$:
\begin{equation}
\bd_1 = 
\begin{pmatrix}
1 \\
1 \\
0 \\
0 \\
\end{pmatrix}, \ \ 
\bd_2 =
\begin{pmatrix}
0 \\
0 \\
1 \\
1\\
\end{pmatrix}, \ \ 
\bd_3 =
\begin{pmatrix}
0 \\
1 \\
1 \\
0 \\
\end{pmatrix}, \ \ 
\bd_4 =
\begin{pmatrix}
1 \\
0 \\
0 \\
1 \\
\end{pmatrix}.
\end{equation}
This indicates the number of extreme vectors: $r = 4$ and
\begin{equation}
r = \text{dim}(\text{ker}(\bW)) +1 = \ell +2,
\end{equation}
where $\ell = 2$.
Moreover, for $p=1, 2$,
\begin{equation}
\text{supp}(\bd_p) = L_p, \ \
\text{supp}(\bd_{\ell +1}) \cap L_p \subsetneq L_p, \ \
\text{supp}(\bd_{\ell +2}) \cap L_p \subsetneq L_p.
\end{equation}
Therefore, we verify Lemma~\ref{lem:deficiency_one_cases part3}.

\bigskip

\begin{figure}[H]
\centering
\includegraphics[scale=0.5]{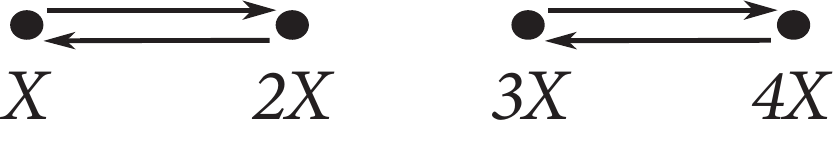}
\caption{A weakly reversible deficiency one mass-action system of Type II from Example \ref{ex:deficiency_not_sum_lemma}}
\label{fig:deficiency_not_sum_lemma}
\end{figure}
\end{example}

\medskip

To conclude this section, we show that given a mass-action system that admits weakly reversible deficiency one realizations, then these realizations must be of the same type.
First, we recall a special result from~\cite{deshpande2022source}:

\begin{theorem}[{\cite[Theorem 6.3]{deshpande2022source}}]
\label{thm:unique of deficiency_one under same linkage}
Consider two weakly reversible mass-action systems $(G, \bk)$ and $(G', \bk')$ having a deficiency of one and the same number of linkage classes. Let $G$ be of Type I and $G'$ be of Type II. Then $(G, \bk)$ and $(G', \bk')$ cannot be dynamically equivalent.
\end{theorem}

After Theorem~\ref{thm:unique of deficiency_one under same linkage}, we are ready to prove the  more general result as follows:

\begin{theorem}\label{lem:unique of deficiency_one}
Given two weakly reversible deficiency one mass-action systems: $(G, \bk)$ of Type I and $(G', \bk')$ of Type II, then $(G, \bk)$ and $(G', \bk')$ cannot be dynamically equivalent.
\end{theorem}

\begin{proof}

For contradiction, assume that the two weakly reversible deficiency one mass-action systems $(G=(V, E), \bk)$ of Type I, and $(G'=(V', E'), \bk')$ of Type II are dynamically equivalent. By Remark~\ref{rmk:dyn_net_vectors}, they have the same set of non-zero net reaction vectors. Using $\text{Im}(\bW) = S$ from Lemma~\ref{lem:wr_ker}, we get that $(G, \bk)$ and $(G', \bk')$ share the same stoichiometric subspace.

Now we claim that $(G, \bk)$ and $(G', \bk')$ have the same number of vertices. For contradiction, suppose there exists a vertex $\by \in V'$ such that $\by \notin V$. Let $\bw_{\by}$ and $\bw'_{\by}$ represent the net reaction vectors corresponding to the vertex $\by$ in $G$ and $G'$. From Remark~\ref{rmk:dyn_net_vectors}, we deduce that 
\begin{equation} \label{bw'y = 0}
\bw'_{\by} = \sum\limits_{\by \rightarrow\by_j\in E'} k'_{\by \rightarrow\by_j}(\by_j - \by) = \mathbf{0}.
\end{equation}
Since the network $G'$ is of Type II, each linkage of $G'$ has deficiency zero. By Proposition~\ref{prop_def_zero_affine}, we get that its vertices are affinely independent within each linkage class. This implies that the reaction vectors $\{ \by_i - \by \}_{\by \rightarrow\by_i \in E'}$ are linearly independent,  contradicting Equation~\eqref{bw'y = 0}.

Assume that there exists a vertex $\by \in L \subseteq V$ such that $\by \notin V'$  (where $L$ is some linkage class in $G$).  Following the steps in the first part, we have that for any $\by' \in V'$, $\bw'_{\by} \neq \mathbf{0}$. This shows that $V' \subsetneq V$. 
Moreover, from $\by \notin V'$, we get that 
\begin{equation} \notag
    \bw_{\by} = \mathbf{0}.
\end{equation}
This implies that the vertices in the linkage class $L$ are not affinely independent. Therefore the deficiency of linkage class $L$ is one. Since $(G, \bk)$ and $(G', \bk')$ have the same stoichiometric subspace and deficiency, we deduce that $G$ has at least one more linkage class than $G'$. From the Pigeonhole Principle, there exists at least one linkage class in $G'$ that is split into different linkage classes in $G$. Let us call this linkage class as $L'_1 \subsetneq V'$. Using Lemma~\ref{lem:dim wr_ker}, we have 
\begin{equation} \notag
\text{dim}(\text{ker}(\bW'_1)) \geq 1.
\end{equation}
where $\bW'_1$ is the matrix of net reaction vectors on $L'_1$. This implies that the stoichiometric subspaces corresponding to the linkage classes in $G$ are not linearly independent, contradicting the fact that $G$ is of Type I. 

Since $(G, \bk)$ and $(G', \bk')$ have the same stoichiometric subspace, number of vertices, and deficiency we obtain that $(G, \bk)$ and $(G', \bk')$ possess the same number of linkage classes. Finally, applying Theorem~\ref{thm:unique of deficiency_one under same linkage}, we get that $(G, \bk)$ and $(G', \bk')$ cannot be dynamically equivalent, which leads to a contradiction.
\end{proof}

The following remark is a direct consequence of Theorem~\ref{lem:unique of deficiency_one}.

\begin{remark}
\label{rmk:unique of deficiency_one}
For any mass-action system $(G, \bk)$, it at most has one type of weakly reversible deficiency one realization, i.e. either Type I or Type II.
\end{remark}

\section{Main results}
\label{sec:algorithms}

This section aims to present the main algorithm of this paper, which checks the existence of a weakly reversible deficiency one realization and outputs one if it exists.
In this algorithm, the inputs are the matrices of source vertices and net reaction vectors via 
\[
\bY_s = (\by_1,\by_2,\ldots,\by_m)
\ \ \text{and } \
\bW = (\bw_1,\bw_2,\ldots,\bw_m)
\]
respectively. 
For the sake of simplicity, we temporarily  let $d_{\bW}=\{\bd_1,\bd_2,\ldots,\bd_r\}$ denote the minimal set of generators of $\text{ker}(\bW)\cap\mathbb{R}^m_{\geq 0}$ in this section.

\medskip

To build the main algorithm, we need an algorithm to search for a weakly reversible realization with a single linkage class.  
We use the algorithm in \cite{def_one} and summarize its main idea as follows.

First, the algorithm in \cite{def_one} checks whether there exists a reaction network realization that generates the given dynamical system such that all the target vertices are among the source vertices, {\em without} imposing the restrictions that $(i)$ the network should be {\em weakly reversible}, and $(ii)$ there should be only {\em one} linkage class. Next, if such a realization exists, the algorithm greedily searches for a maximal realization (a realization containing the maximum number of reactions) that generates the same dynamical system, while still imposing the restriction that all target vertices are among the source vertices. 
The algorithm uses the fact that if the initial realization was weakly reversible and consisted of a single linkage class, then the maximal realization found using this procedure preserves weakly reversibility and a single linkage class. 
Finally, based on this maximal realization, the algorithm constructs a Kirchoff matrix $Q$ and checks whether $\dim(\text{ker}(Q))=1$ and $\text{supp}(\text{ker}(Q))=\{1,2,3,\ldots,m\}$. 
If both conditions are satisfied, then the maximal realization is weakly reversible and consists of a single linkage class. Otherwise,  there is no such realization that generates the given polynomial dynamical system. 

For more details on this algorithm and its implementation and complexity, please see~\cite{def_one}. In what follows, we will refer to the algorithm in~\cite{def_one} as \textbf{Alg-WR}$^{\ell=1}$.

\subsection{Algorithm for weakly reversible and deficiency one realization}

Now we state the main algorithm. The key idea is to find a proper decomposition on $\bW = \bY \bA_{\bk}$, which allows a weakly reversible and deficiency one realization. We apply 
\textbf{Alg-WR}$^{\ell=1}$ to ensure weakly reversibility and  the single linkage class condition, and use results in Section~\ref{sec:pointed_cone} to guarantee that the deficiency of the network is one. 

\bigskip\bigskip


\begin{breakablealgorithm}
\caption{(Check the existence of a weakly reversible deficiency one realization)} \label{algorithm:WR_def_one}

\medskip

\noindent \textbf{Input:} The matrices of source vertices $\bY_s = (\by_1, \ldots, \by_m)$, and net reaction vectors $\bW = (\bw_1, \ldots, \bw_m)$ that generate the dynamical system $\dot{\bx} = \sum\limits_{i=1}^m \bx^{\by_i} \bw_i$. 

\medskip

\noindent \textbf{Output:} A weakly reversible deficiency one realization if exists or output that it does not exist. 

\medskip

\begin{algorithmic}[1] 

\State Set flag = 0 and $\text{dim}(\text{ker}(\bW)) = \bw^*$;

\State Find the minimal set of generators $d_{\bW}=\{\bd_1,\bd_2,\ldots,\bd_r\}$ of the pointed cone $\text{ker}(\bW)\cap\mathbb{R}^m_{\geq 0}$. 

\If{$r < 2$ or
$\displaystyle\bigcup_{i=1}^r \text{supp} (\bd_i) \neq [m]$:}
\State\label{line:5} Exit the main program; 

\ElsIf{r = 2}\label{line:6}

\State Pass $\bY_s, \bW$ through 
\textbf{Alg-WR}$^{\ell=1}$
\If 
{\textbf{Alg-WR}$^{\ell=1}$ outputs that a weakly reversible realization consisting of a single linkage class does not exist}\label{line:8} 

\State Exit the main program;

\Else

\State flag=1;

\State Exit the main program; 

\EndIf

\Else\label{line:13}

\For{$i = 1, 2, \ldots, r-1$}        
\For{$j = i+1, i+2, \ldots, r$}\label{line:loop}               
\State $S_1 = \{\bd_i, \bd_j\}$.

\State $S_2 = \bd_{\bW} \setminus S_1$ and \footnote{For the simplicity of notations, we use a new symbol $\hbd$ to represent the vectors in $S_2$}set $S_2 := \{ \hbd_p \}^{r-2}_{p=1}$.

\If {$r = \bw^*$ and the support of $S_1$ and every member of $S_2$ are disjoint}\label{line:19}

\State Define linkage classes to be $\{L_p\}^{r-1}_{p=1}$, where $L_p:= \{\text{supp}(\hbd_p), \hbd_p \in S_2\}$ for $1 \leq p \leq r-2$, and $L_{r-1} := \{\text{supp}(\bd_i) \cup \text{supp}(\bd_j) \}$. 

\State Let $\bY_p$ denotes the vertices in linkage class $L_p$, and $\bW_p$ denotes the matrix of net reaction vectors corresponding to $\bY_p$.


\For{$p=1$ to $r-1$}

\State Pass $\bY_p,\bW_p$ through 
\textbf{Alg-WR}$^{\ell=1}$

\If 
{\textbf{Alg-WR}$^{\ell=1}$ outputs that a weakly reversible realization consisting of a single linkage class does not exist}\label{line:23}

\State Go to line~\ref{line:loop};





\EndIf

\EndFor


\State\label{line:27} flag=2;

\State Exit the main program;


\ElsIf {$r = \bw^* + 1$ and the support of the members of $S_2$ partition $[m]$}\label{line:29}

\State Define linkage classes to be $\{L_p\}^{r-2}_{p=1}$, where $L_p:= \{\text{supp}(\hbd_p), \hbd_p \in S_2\}$ for $1 \leq p \leq r-2$.

\State Let $\bY_p$ denotes the vertices in linkage class $L_p$, and $\bW_p$ denotes the matrix of net reaction vectors corresponding to $\bY_p$.


    
\If {$\text{dim}(\text{ker}(\bW_1)) = \text{dim}(\text{ker}(\bW_2))  = \cdots = \text{dim}(\text{ker}(\bW_{r-2}))  =  1$}\label{line:31}

\For{$p=1$ to $r-2$}

\State Pass $\bY_p,\bW_p$ through
\textbf{Alg-WR}$^{\ell=1}$

\If 
{\textbf{Alg-WR}$^{\ell=1}$ outputs that a weakly reversible realization consisting of a single linkage class does not exist}\label{line:34}

\State Go to line~\ref{line:loop};

\EndIf

\EndFor

\State flag=3;

\State Exit the main program;

\EndIf

\EndIf


\EndFor

\EndFor

\EndIf

\bigskip

\State \textbf{End of main program}

\vspace{5mm}

\If {flag = 0}

\State\label{line:flag=0} Print: No weakly reversible and deficiency one realization exists.

\ElsIf{flag = 1}

\State\label{line:flag=1} Print: Weakly reversible and deficiency one realization consisting of a single linkage class exists.

\ElsIf{flag = 2}

\State\label{line:flag=2} Print: Weakly reversible and deficiency one realization of Type I exists.

\ElsIf{flag = 3}

\State\label{line:flag=3} Print: Weakly reversible and deficiency one realization of Type II exists.

\EndIf



\end{algorithmic}

\end{breakablealgorithm}

\newpage

\tikzset{%
    Node/.style={rectangle, rounded corners, draw=black, thick, fill=blue!10, fill opacity = 1, minimum width=4.5cm, minimum height=1cm, outer sep=0pt},
    Edge/.style={very thick, style={very thick, arrows={-Stealth[length=7.5pt,width=7.5pt]}}},
}

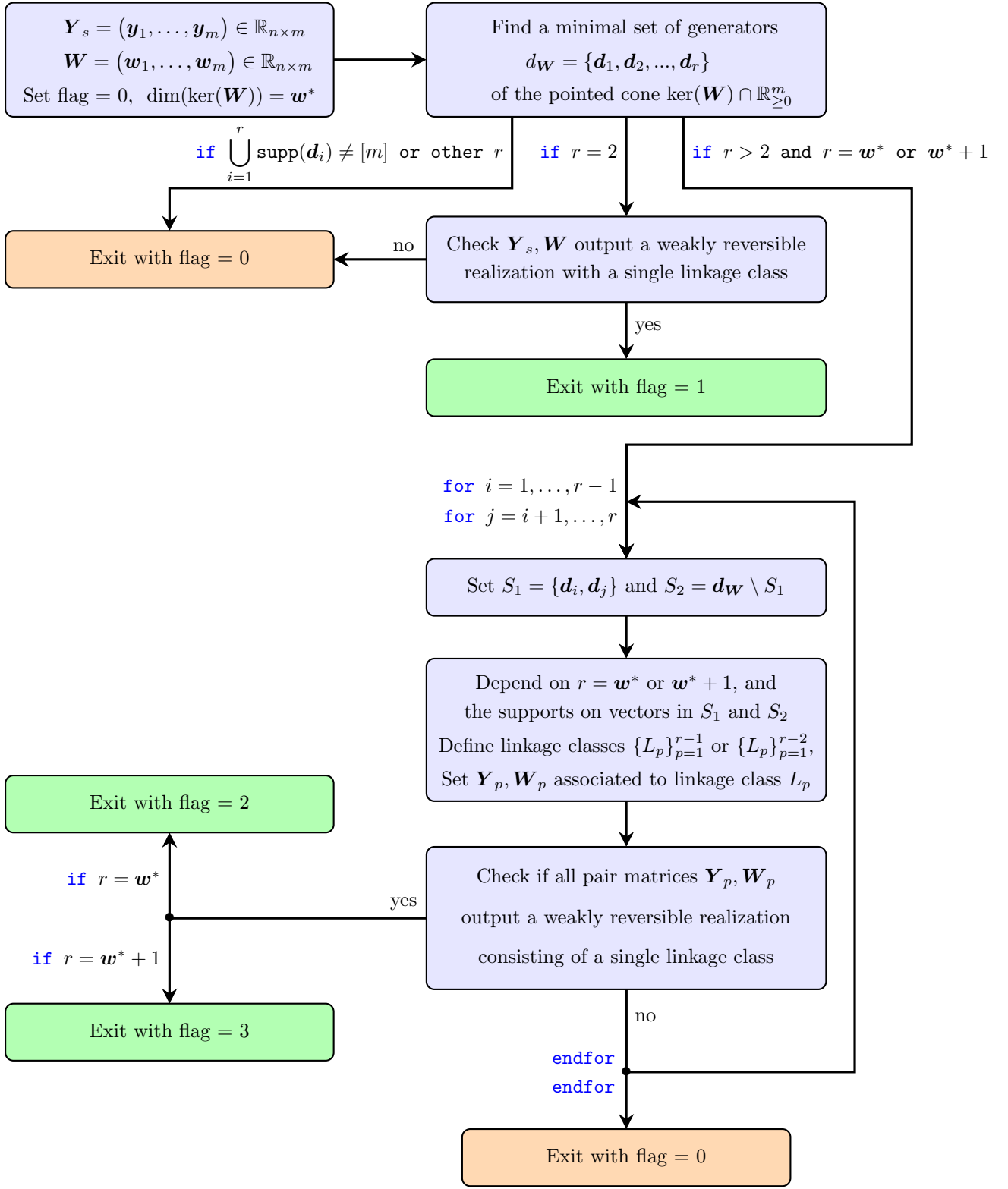
\begin{figure}[H]
\centering
    \begin{tikzpicture}
    \node[Node,  minimum width=5.75cm, fill=orange!30] (fail1) at (1,-3.75) {\small Exit with flag = 0};
        \node[Node,  minimum width=5.75cm, fill=orange!30] (fail2) at (9, -19.5) {\small Exit with flag = 0};
    \node[Node, minimum height=2cm, minimum width=5.75cm] (YW) at (1, -0.25) {};
        \node at (-1,0.3) [right] {\small $\bY_s = \begin{pmatrix} \by_1 ,  \ldots ,  \by_m\end{pmatrix} \in  \mathbb{R}_{n \times m}$};
        \node at (-1,-0.3) [right] {\small $\bW = \begin{pmatrix} \bw_1 ,  \ldots ,  \bw_m\end{pmatrix} \in \mathbb{R}_{n \times m}$};
        \node at (-1.7,-0.9) [right] {\small Set flag = 0, \ $\rm{dim}(\rm{ker}(\bW)) = \bw^*$};
    \node[Node, minimum height=2cm, minimum width=7cm] (dw) at (9,-0.25) {};
        \node at (6.5,0.3) [right] {\small Find a minimal set of generators};
        \node at (7.1,-0.3) [right] {\small $d_{\bW}=\{\bd_1,\bd_2,...,\bd_r\}$};
        \node at (6.5,-0.9) [right] {\small of the pointed cone $\text{ker}(\bW)\cap\mathbb{R}^m_{\geq 0}$};
        \draw[Edge] (YW)--(dw) node {};   
        \draw[Edge] (7,-1.25) to (7,-2.5) node [above=17pt,left] {\small \ttfamily\bfseries \textcolor{blue}{if} $\displaystyle\bigcup_{i=1}^r \text{supp} (\bd_i) \neq [m]$ or other $r$} to (1,-2.5)  to (fail1);
    \node[Node,  minimum width=7cm, minimum height=1.5cm] (subalg) at (9,-3.75) {};
        \node at (9,-3.5) {\small Check $\bY_s, \bW$ output a weakly reversible};
        \node at (9,-4) {\small realization with a single linkage class};
        \draw[Edge] (dw)--(subalg) node [above= 54pt, left] {\small \ttfamily\bfseries \textcolor{blue}{if} $r = 2$};
        \draw[Edge] (subalg) node[xshift=-3.9cm, above] {\small no}--(fail1);
    \node[Node,  minimum width=7cm, minimum height=1cm, fill=green!30] (flag1) at (9,-6) {};
        \node at (9,-6) {\small Exit with flag = 1};
    \draw[Edge] (subalg)--(flag1) node [above=30pt, right] {\small yes};
    \node[Node, minimum height=1cm, minimum width=7cm] (S1S2) at (9,-9.5) {};
        \node at (9,-9.5) {\small Set $S_1 = \{\bd_i, \bd_j\}$ and $S_2 = \bd_{\bW}\setminus S_1$};
        \draw[Edge] (10,-1.25) to (10,-2.5) node [above = 18pt,right] {\small \ttfamily\bfseries \textcolor{blue}{if} $r > 2$ and $r = \bw^*$ or $\bw^* + 1$} to (14,-2.5)  to (14,-7) to (9,-7) to (S1S2);
        \draw[Edge] (9,-7)--(S1S2) node [above = 50pt, left] {\small \ttfamily\bfseries \textcolor{blue}{for} $i = 1, \ldots, r-1$ };   
        \draw[Edge] (9,-7)--(S1S2) node [above = 35pt, left] {\small \ttfamily\bfseries \textcolor{blue}{for} $j = i+1, \ldots, r$ };
    \node[Node, minimum height=2.5cm, minimum width=7cm] (Lp) at (9,-12) {};
        \node at (9,-11.2) {\small Depend on $r = \bw^*$ or $\bw^* + 1$, and};
        \node at (9,-11.7) {\small the supports on vectors in $S_1$ and $S_2$};
        \node at (9,-12.3) {\small Define linkage classes $\{L_p\}^{r-1}_{p=1}$ or $\{L_p\}^{r-2}_{p=1}$,};
        \node at (9,-12.9) {\small Set $\bY_p, \bW_p$ associated to linkage class $L_p$};
        \draw[Edge] (S1S2)--(Lp) node {\small };
    \node[Node,  minimum height=2.5cm, minimum width=7cm] (YpWp) at (9,-15.3) {};
        \node at (9,-14.6) {\small Check if all pair matrices $\bY_p, \bW_p$};
        \node at (9,-15.3) {\small output a weakly reversible realization};
        \node at (9,-16.0) {\small consisting of a single linkage class};
        \draw[Edge] (Lp)--(YpWp) node[below] {};
        \node at (9,-18) {$\bullet$};
        \node at (9,-17.75) [left] {\small \ttfamily\bfseries \textcolor{blue}{endfor}\,};
        \node at (9,-18.25) [left] {\small \ttfamily\bfseries \textcolor{blue}{endfor}\,};
        \draw[Edge] (YpWp) to (9,-18) to (13,-18) to (13,-8) to (9, -8);
    \node[Node,  minimum width=5.75cm, fill=green!30] (flag2) at (1,-13.3) {\small Exit with flag = 2};
        \node[Node,  minimum width=5.75cm, fill=green!30] (flag3) at (1,-17.3) {\small Exit with flag = 3};
        \draw[Edge] (YpWp) node [xshift=-3.9cm,above]  {\small yes} to (5.5,-15.3) to (1,-15.3) node [above= 20pt,left] {\small \ttfamily\bfseries \textcolor{blue}{if} $r = \bw^*$} to (flag2);
        \draw[Edge] (YpWp) to (5.5,-15.3) to (1,-15.3)node [below= 20pt,left] {\small \ttfamily\bfseries \textcolor{blue}{if} $r = \bw^*+1$} to (flag3);
        \node at (1,-15.3) {$\bullet$};
        \draw[Edge] (YpWp)--(fail2) node [above=70pt, right] {\small no};; 
\end{tikzpicture}  
\caption{Algorithm \ref{algorithm:WR_def_one} for finding a weakly reversible deficiency one realization that generates a given polynomial dynamical system $\dot{\bx} = \sum\limits_{i=1}^m \bx^{\by_i} \bw_i$. 
}
\label{fig:Alg} 
\end{figure}

\newpage

Now we show the correctness of Algorithm \ref{algorithm:WR_def_one} via the following two lemmas.

\begin{lemma}
\label{lem: algorithm_2_part a}

Suppose Algorithm \ref{algorithm:WR_def_one} exits with a positive flag value, then there exists a weakly reversible deficiency one realization of the dynamical system $\dot{\bx} = \sum\limits_{i=1}^m \bx^{\by_i} \bw_i$. Moreover, we have
\begin{enumerate}
\item[(a)] If flag = 1, the system admits a weakly reversible deficiency one realization consisting of a single linkage class.

\item[(b)] If flag = 2, the system admits a weakly reversible deficiency one realization of Type I.

\item[(c)] If flag = 3, the system admits a weakly reversible deficiency one realization of Type II.
\end{enumerate}
\end{lemma}

\begin{proof}

$(a)$
From flag $= 1$, we obtain that $r = 2$ with $\text{supp}(\bd_1) \cup \text{supp}(\bd_2) = [m]$. Moreover, the input matrices $\bY_s$ and $\bW$ pass through \textbf{Alg-WR}$^{\ell=1}$.

Then there exists a weakly reversible realization with a single linkage class that generates the dynamical system $\dot{\bx} = \sum\limits_{i=1}^m \bx^{\by_i} \bw_i$. Using Remark~\ref{rmk:generator_on_deficiency_one}, we conclude its deficiency is one.

\medskip

$(b)$
From flag $= 2$, we get $r = \text{dim} (\text{ker}(\bW)) > 2$, and $r-1$ linkage classes $\{L_1, \ldots, L_{r-1}\}$ as follows.
There exists some $1 \leq i < j \leq r$,
\begin{equation} \notag
S_1 = \{\bd_i,\bd_j\} \ \ \text{and } \
S_2 = d_{\bW} \setminus S_1.
\end{equation}
For the simplicity of notations, we rename $r-2$ vectors in $S_2$ as 
 $S_2 := \{ \hbd_p \}^{r-2}_{p=1}$ and set
\begin{equation} \notag
\begin{split}
& L_p = \{\text{supp}(\hbd_p): \ \hbd_p \in S_2 \}, \ \text{for} \ 1 \leq p \leq r-2,
\\& L_{r-1} =\{\text{supp}(\bd_i) \cup \text{supp}(\bd_j) \},
\end{split}
\end{equation}
with $\{L_1, L_2, \ldots, L_{r-1} \}$ partition $[m]$.

Moreover, for any $1 \leq q \leq r-1$, the 
matrices of source vertices and net reaction vectors $\bY_q, \bW_q$ related to the linkage class $L_q$ pass through 
\textbf{Alg-WR}$^{\ell=1}$.
Thus each linkage class $L_q$ admits a weakly reversible realization.
Together with $\{L_q\}^{r-1}_{q=1}$ have disjoint supports in $[m]$, we have
\begin{equation}
\begin{split} \label{flag=2 ker estimate}
& \text{ker}(\bW_p) = \{\text{span}(\hbd_p) \}, \ \text{for} \ 1 \leq p \leq r-2,
\\& \text{ker}(\bW_{r-1}) = \text{span} \{ \bd_i, \bd_j \}.
\end{split}
\end{equation}
Using Lemma~\ref{lem:deficiency_zero generator} and Remark~\ref{rmk:generator_on_deficiency_one} on the realization under \textbf{Alg-WR}$^{\ell=1}$,
we get
\begin{equation} \label{flag=2 deficiency}
\delta_1 = \cdots = \delta_{r-1} = 0 
\ \ \text{and } \
\delta_{r-1} = 1,
\end{equation}
where $\delta_q$ represents the deficiency of linkage class $L_q$.

Now we compute the deficiency of the whole realization $\delta$. From \eqref{flag=2 ker estimate}, we obtain
\begin{equation} \label{flag=2 ker sum}
\text{dim}(\text{ker}(\bW_1) + \text{dim}(\text{ker}(\bW_2) + \cdots + \text{dim}(\text{ker}(\bW_{r-1})
= r = \text{dim}(\text{ker}(\bW).
\end{equation}
Applying Lemma~\ref{lem:wr_ker} and Lemma \ref{lem:dim wr_ker} on Equation~\eqref{flag=2 ker sum}, we deduce for $p = 1, \ldots, r-1$, 
\begin{equation} \notag
\text{dim}(\text{ker}(\bW_q)) = 1 + \delta_{q} = |L_q| - s_q \ \ \text{and } \
\sum\limits_{q=1}^{r-1} s_{q} = s,
\end{equation}
where $s_q$ and $s$ represent the stoichiometric subspace for linkage class $L_q$ and whole network respectively. 
Then we do the summation from $q = 1$ to $r-1$, and get
\begin{equation} \notag
\sum\limits_{q=1}^{r-1} (|L_q| - s_q) = m - s = (r-1) + \sum\limits_{q=1}^{r-1} \delta_{q}.
\end{equation}
From Equation~\eqref{flag=2 deficiency}, we conclude that
\begin{equation} \notag
\delta = m - s - (r-1) = \sum\limits_{q=1}^{r-1} \delta_{q} = 1,
\end{equation}
and the system admits a weakly reversible deficiency one realization of Type I.

\medskip

$(c)$
From flag $= 3$, we get $r = \text{dim} (\text{ker}(\bW)) + 1 > 2$, and $r-2$ linkage classes $\{L_1, \ldots, L_{r-2}\}$ as follows.
There exists some $1 \leq i < j \leq r$,
\begin{equation} \notag
S_1 = \{\bd_i,\bd_j\} 
\ \ \text{and } \
S_2 = d_{\bW} \setminus S_1,
\end{equation}
Similarly, we rename $r-2$ vectors in $S_2$ as $S_2 := \{ \hbd_p \}^{r-2}_{p=1}$ and set
\begin{equation} \notag
L_p = \{\text{supp}(\hbd_p): \ \hbd_p \in S_2 \}, \ \text{for} \ 1 \leq p \leq r-2,
\end{equation}
with $\{L_1, L_2, \ldots, L_{r-2} \}$ partition $[m]$.

Moreover, for any $1 \leq p \leq r-2$,  the 
matrices of source vertices and net reaction vectors $\bY_p, \bW_p$
related to the linkage class $L_p$ pass through \textbf{Alg-WR}$^{\ell=1}$.
Thus each linkage class $L_q$ admits a weakly reversible realization with $\text{dim}(\text{ker}(\bW_p)) = 1$.
Applying that $\{L_p\}^{r-2}_{p=1}$ partition $[m]$, we have
\begin{equation}
\label{flag=3 ker estimate}
\text{ker}(\bW_p) = \{\text{span}(\hbd_p) : \hbd_p \in S_2 \}, \ \text{for} \ 1 \leq p \leq r-2.
\end{equation}
Using Lemma \ref{lem:deficiency_zero generator} on the realization under 
\textbf{Alg-WR}$^{\ell=1}$,
we get
\begin{equation} \label{flag=3 deficiency}
\delta_1 = \cdots = \delta_{r-2} = 0,
\end{equation}
where $\delta_p$ represents the deficiency of linkage class $L_p$.

Now we compute the deficiency of the whole realization $\delta$. From \eqref{flag=3 ker estimate}, we obtain
\begin{equation} \label{flag=3 ker sum}
\text{dim}(\text{ker}(\bW_1) + \text{dim}(\text{ker}(\bW_2) + \cdots + \text{dim}(\text{ker}(\bW_{r-2})
= r-2 = \text{dim}(\text{ker}(\bW) - 1.
\end{equation}
Applying Lemma~\ref{lem:wr_ker} and Lemma \ref{lem:dim wr_ker} on Equation~\eqref{flag=3 ker sum}, we deduce for $p = 1, \ldots, r-2$, 
\begin{equation} \notag
\text{dim}(\text{ker}(\bW_p)) = 1 + \delta_{p} = |L_p| - s_p 
\ \ \text{and } \ 
\sum\limits_{p=1}^{r-2} s_{p} = s + 1,
\end{equation}
where $s_p$ and $s$ represent the stoichiometric subspace for linkage class $L_p$ and whole network respectively. 
Summing from $p = 1$ to $r-2$, we get
\begin{equation} \notag
\sum\limits_{p=1}^{r-1} (|L_p| - s_p) = m - (s + 1) = (r-2) + \sum\limits_{q=1}^{r-1} \delta_{q}.
\end{equation}
From Equation~\eqref{flag=3 deficiency}, we conclude that
\begin{equation} \notag
\delta = m - s - (r-2) = \sum\limits_{q=1}^{r-1} \delta_{q} + 1 = 1,
\end{equation}
and the system admits a weakly reversible deficiency one realization of Type II.
\end{proof}

\begin{lemma}
\label{lem: algorithm_2_part b}
Suppose the dynamical system $\dot{\bx} =\displaystyle \sum_{i=1}^m \bx^{\by_i} \bw_i$ admits a weakly reversible deficiency one realization, then Algorithm \ref{algorithm:WR_def_one} must set the flag value to be either $1,2$ or $3$.
\end{lemma}

\begin{proof}

Note that every weakly reversible deficiency one network belongs to the following:
\begin{enumerate}
\item Weakly reversible deficiency one realization consisting of a single linkage class.
\item Weakly reversible deficiency one realization of Type I, with two or more  linkage classes.
\item Weakly reversible deficiency one realization of Type II.
\end{enumerate}
Therefore, we split our proof into the above three cases.

\medskip

\textbf{Case 1: }
Suppose the system admits a weakly reversible deficiency one realization consisting of a single linkage class.
From Lemma \ref{lem:deficiency_one_cases part1}, we obtain that $r = 2$ and the input $\bY$ and $\bW$ pass through 
\textbf{Alg-WR}$^{\ell=1}$
from the weakly reversibility. Therefore, Algorithm \ref{algorithm:WR_def_one} will exit with flag $=1$.

\medskip

\textbf{Case 2: }
Suppose the system admits a weakly reversible deficiency one realization of Type I with $\ell > 1$ linkage classes, denoted by $L_1, L_2, \ldots, L_{\ell}$.
From Lemma \ref{lem:deficiency_one_cases part2}, we have  
\begin{equation} \notag
\text{dim}(\text{ker}(\bW_{\ell}))= 2
\ \ \text{and } \ 
\text{dim}(\text{ker}(\bW_p)) = 1, \ \text{for} \ 1 \leq p \leq \ell-1,
\end{equation}
and
\begin{equation} \notag
\text{dim}(\text{ker}(\bW)) = \text{dim}(\text{ker}(\bW_1)) + \cdots + \text{dim}(\text{ker}(\bW_{\ell})) = \ell + 1 = r.
\end{equation}
Moreover, there exist $\ell + 1$ vectors  $\bd_1, \ldots, \bd_{\ell + 1}$ forming the minimal set of generators of $\text{ker}(\bW)\cap\mathbb{R}^m_{\geq 0}$, such that for
$p=1, \ldots, \ell-1$,  
\begin{equation} \notag
\text{supp}(\bd_p) = L_p,
\end{equation}
and
\begin{equation} \notag
\text{supp}(\bd_{\ell}) \subsetneq L_{\ell}, \ \
\text{supp}(\bd_{\ell+1}) \subsetneq L_{\ell}, \ \ 
\text{supp}(\bd_{\ell})\cup \text{supp}(\bd_{\ell +1}) = L_{\ell}.
\end{equation}
Thus, when $i = \ell$ and $j = \ell +1$, 
(i.e. $\bd_i = \bd_{\ell}$ and $\bd_j = \bd_{\ell+1}$), Algorithm \ref{algorithm:WR_def_one} will exit with flag $=2$.

\medskip

\textbf{Case 3: }
Suppose the system admits a weakly reversible deficiency one realization of Type II with $\ell > 1$ linkage classes, denoted by $L_1, L_2, \ldots, L_{\ell}$.
From Lemma \ref{lem:deficiency_one_cases part3}, we have  
\begin{equation} \notag
\text{dim}(\text{ker}(\bW_p)) = 1, \ \text{for} \ 1 \leq p \leq \ell,
\end{equation}
and
\begin{equation} \notag
\text{dim}(\text{ker}(\bW)) - 1 = \text{dim}(\text{ker}(\bW_1)) + \cdots + \text{dim}(\text{ker}(\bW_{\ell})) = \ell = r - 2.
\end{equation}
Moreover, there exist $\ell + 2$ vectors  $\bd_1, \ldots, \bd_{\ell + 2}$ forming the minimal set of generators of $\text{ker}(\bW)\cap\mathbb{R}^m_{\geq 0}$, such that for
$p=1, \ldots, \ell$,
\begin{equation} \notag
\text{supp}(\bd_p) = L_p, \ \
\text{supp}(\bd_{\ell +1}) \cap L_p \subsetneq L_p, \ \
\text{supp}(\bd_{\ell +2}) \cap L_p \subsetneq L_p.
\end{equation}
Again when we pick $i = \ell$ and $j = \ell +1$, Algorithm \ref{algorithm:WR_def_one} will exit with flag $=3$.

\medskip

Lastly, we show every mass-action system admitting a weakly reversible deficiency one realization has a unique flag value after applying Algorithm \ref{algorithm:WR_def_one}.
Following Remark \ref{rmk:unique of deficiency_one}, we deduce that 
if flag $=3$ after passing the same mass-action system through the algorithm, the flag value cannot equal $1$ or $2$.
From Lemma \ref{lem:deficiency_zero generator} and Lemma \ref{lem:deficiency_one_cases part2}, we have $r = 2$ if flag $=1$, and $r = \ell + 1 > 2$ if flag $=2$. 
Thus, it is also impossible that the flag equals both $1$ and $2$ on the same mass-action system.
Therefore, we show the uniqueness and prove this lemma.
\end{proof}

The following remark is a direct consequence of Lemma~\ref{lem: algorithm_2_part b}.

\begin{remark}
If Algorithm \ref{algorithm:WR_def_one} sets the value of flag to 0, then $\dot{\bx} = \sum\limits_{i=1}^m \bx^{\by_i} \bw_i$ does not admit a weakly reversible deficiency one realization.
\end{remark}

\medskip

\begin{example}
\label{ex: alg2 pass}

Consider the system of differential equations
\begin{equation} \label{eq: alg2 pass}
\begin{split}
\dot{x} & = - x +x^2, \\
\dot{y} & =  x^3 - x^3y^2.
\end{split}
\end{equation}
We have $n = 2$ for the two state variables, and $m = 5$ for the five distinct monomials. The matrices of source vertices and net direction vectors are
\begin{equation}
\bY_s =  \begin{pmatrix}
1  & 2  & 3 & 3 &3\\
0 	& 0 & 0 & 1  & 2
\end{pmatrix}, \ \text{and } \
\bW =  \begin{pmatrix}
1  & -1  & 0 & 0 & 0\\
0 	& 0 & 1 & 0 & -1
\end{pmatrix}.
\end{equation}
respectively, which are inputs to Algorithm \ref{algorithm:WR_def_one}.

Then, we can compute that $\text{dim}(\text{ker}(\bW)) = 3$, and extreme vectors of $\text{ker}(\bW)\cap\mathbb{R}^m_{\geq 0}$ is given by
\begin{equation} \notag
\bd_1 = 
\begin{pmatrix}
1  \\
1 \\
0  \\
0 \\
0 \\
\end{pmatrix}, \ \ 
\bd_2=\begin{pmatrix}
0  \\
0 \\
0  \\
1 \\
0 \\
\end{pmatrix}, \ \ 
\bd_3=\begin{pmatrix}
0  \\
0 \\
1  \\
0 \\
1 \\
\end{pmatrix}.
\end{equation}
This shows that $r =3$, and the algorithm enters line~\ref{line:13}. 

Next, when we pick $i = 2$, $S_1= \{\bd_2,\bd_3\}$ and $S_2= \{\bd_1\}$. Note that $r = \text{dim}(\text{ker}(\bW)) = 3$, and the support of $S_1$ and every member of $S_2$ are disjoint, the algorithm defines candidate linkage classes are follows: 
\begin{equation} \notag
L_1 = \{\text{supp}(\bd_1)\} = \{1, 2\}, \ \
L_2 = \{\text{supp}(\bd_2) \cup \text{supp}(\bd_3) \} = \{3, 4, 5\}.
\end{equation}
Following the candidate linkage classes $L_1, L_2$, we derive the corresponding matrices of source vertices and net direction vectors:
\begin{equation} \notag
\begin{split}
& \bY_1 =  \begin{pmatrix}
1  & 2 \\
0  & 0 
\end{pmatrix}, \ \ 
\bW_1 =  \begin{pmatrix}
1  & -1  \\
0  &  0
\end{pmatrix}, \ \text{and } \ 
\bY_2 =  \begin{pmatrix}
3 & 3 & 3 \\
0 & 1 & 2 
\end{pmatrix}, \ \
\bW_2 =  \begin{pmatrix}
0  &  0 & 0 \\
1  &  0 & -1
\end{pmatrix}.
\end{split}
\end{equation}

After that, we pass two pairs $(\bY_1, \bW_1)$ and $(\bY_2, \bW_2)$ through 
\textbf{Alg-WR}$^{\ell=1}$.
Both pairs pass successfully through 
\textbf{Alg-WR}$^{\ell=1}$,
i.e., a weakly reversible single linkage class exists for both arrangements. Finally, the algorithm sets flag $=2$ on line~\ref{line:27}, and exits. 
Therefore, \eqref{eq: alg2 pass} admits a weakly reversible deficiency one realization of Type I, whose E-graph is shown in Figure \ref{fig:alg2pass}.

\begin{figure}[H]
\centering
\includegraphics[scale=0.5]{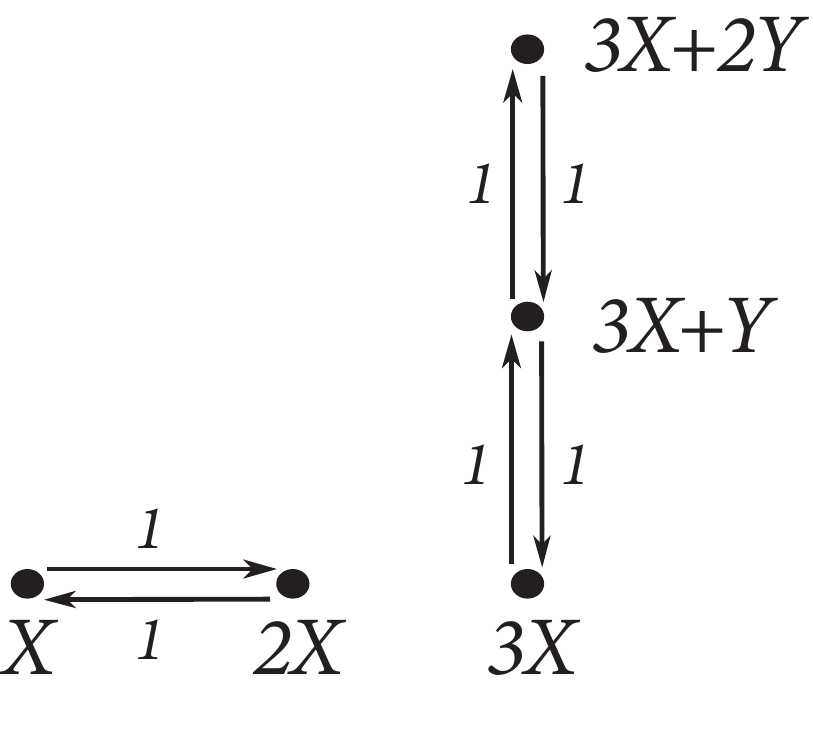}
\caption{A weakly reversible deficiency one mass-action system from Example \ref{ex: alg2 pass}}
\label{fig:alg2pass}
\end{figure}
\end{example}

\begin{example}

Consider the system of differential equations
\begin{equation} \label{ex: alg2 fail}
\begin{split}
\dot{x} & = - x +x^2, \\
\dot{y} & =  0.
\end{split}
\end{equation}
We have $n = 2$ for the two state variables, and $m = 2$ for the two distinct monomials. The matrices of source vertices and net direction vectors are
\begin{equation}
\bY_s =  \begin{pmatrix}
1  & 2  \\
0 	& 0 
\end{pmatrix}, \ \text{and } \
\bW = \begin{pmatrix}
-1  & 1 \\
0 	& 0 
\end{pmatrix}.
\end{equation}
respectively, which are inputs to Algorithm \ref{algorithm:WR_def_one}.

Then, we compute that $\text{dim}(\text{ker}(\bW)) = 1$, and the extreme vector of $\text{ker}(\bW)\cap\mathbb{R}^2_{\geq 0}$ is
\begin{equation} \notag
\bd_1 = 
\begin{pmatrix}
1  \\
1 
\end{pmatrix}.
\end{equation}
This shows that $r=1$, then the algorithm satisfies the condition on line~\ref{line:5} and exits the program with initial flag $= 0$.
Therefore, there doesn't exist any weakly reversible deficiency one realization for this system.
\end{example}

\subsection{Implementation of Algorithm \ref{algorithm:WR_def_one}}
\label{sec:complexity}

In this section, we discuss how to implement Algorithm \ref{algorithm:WR_def_one}. 
The algorithm is designed to find a weakly reversible deficiency one realization that generates the dynamical system $\dot{\bx} = \sum\limits_{i=1}^m \bx^{\by_i} \bw_i$, and it has three key steps:

\begin{enumerate}

\item Compute $\dim (\ker(\bW))$ and $\dim (\ker(\bW_i))$.

\item Find the extreme vectors of the cone $\text{ker}(\bW)\cap\mathbb{R}^m_{\geq 0}$.

\item Pass pairs of the matrices
$\bY_s, \bW$ or $\bY_i, \bW_i$ through \textbf{Alg-WR}$^{\ell=1}$.

\end{enumerate}

\medskip

In Step 1, the implementation needs a rank-revealing factorization; we need to find a basis of $\bW$ or $\bW_{i}$, and then we can check the number of vectors in this basis.
This is equivalent to solving a linear programming problem. 

\medskip

In Step 2, we note that by the Minkowski-Weyl theorem~\cite{2016david,rockafellar1970convex}, there exists two representations of a polyhedral cone $C$ given by: 
\begin{enumerate}

\item[(a)] \textbf{H-representation}: 
There exists a matrix $A$, such that the cone $C$ can be written as
\begin{equation*}
    C = \{A x\leq 0\}.
\end{equation*}

\item[(b)] \textbf{V-representation}: 
The cone $C$ has the minimal set of generators $\{d_i\}$, such that
\begin{equation*}
    C = \displaystyle\sum_{i=1}^r \lambda_i d_i,
\end{equation*}
where $\lambda_i \geq 0$.
\end{enumerate}


To find the extreme vectors of the cone $\text{ker}
(\bW)\cap\mathbb{R}^m_{\geq 0}$,
we need a way to convert from the H-representation to the V-representation. There are two popular ways of performing this conversion:

\begin{enumerate}

\item[(a)] \emph{Double description method}: This is an example of an \emph{incremental} method, where the conversion from H-representation to V-representation is performed assuming that the solution to a smaller problem is already known~\cite{motzkin1953double}. In particular, let  $C(A):=\{A x\leq 0\}$. Let $J$ be a subset of the row indices of $A$. We will denote by $A_J$ the submatrix of $A$ obtained by selecting the $J$ rows of $A$. Let us assume that we have found the minimal set of generators for the cone $C(A_J)$. We will denote by $E$ the generating matrix whose columns are the extreme vectors of $C(A_J)$.  The double description algorithm selects an index $h$ that is not present in $J$ and constructs the generating matrix $E'$ that corresponds to the $A_{J+h}$. This is repeated for several iterations until the generating matrix for $C(A)$ is found. This algorithm is useful in cases where the inputs are degenerate and the dimension of the cone is small.

\item[(b)] \emph{Pivoting methods}: In this method, the extreme vectors of the cone are found by the \emph{reverse search} technique, where the simplex algorithm (that uses pivots iteratively) is run in reverse for the linear programming problem $Ax\leq 0$. The reverse search method determines the extreme vectors of the cone by building a tree in a depth-first-search fashion. This method was developed by Avis and Fukuda~\cite{avis1991pivoting}. It is particularly useful for non-degenerate inputs where it runs in time polynomial of the input size. 
\end{enumerate}

\medskip

In Step 3, we apply \textbf{Alg-WR}$^{\ell=1}$, and this step can be done by solving a sequence of linear programming problems. More details can be found in section 4.4 in~\cite{def_one}.

\section{Discussion}
\label{sec:discussion}

Weakly reversible deficiency one networks are ubiquitous in biochemistry, and are known to have the capacity to exhibit sophisticated dynamics. Some notable examples include the Edelstein network, as in Example~\ref{ex:edelstein}.
To better understand their dynamics, we divide them into two categories:  (i) Type I networks, where all linkage classes have deficiency zero except one linkage class having deficiency one, and (ii)  Type II networks, where all linkage classes have deficiency zero. The crucial quantity in the analysis of such networks is the pointed cone $\text{ker}(\bW)\cap\mathbb{R}^m_{\geq 0}$, where $\bW$ is the matrix formed by the net reaction vectors. In particular, the extreme vectors of this cone can be divided into two classes: cyclic generators and stoichiometric generators. Networks of Type I possess only cyclic generators and satisfy the conditions of the Deficiency One Theorem. Consequently, for Type I networks, there exists a unique steady state within every stoichiometric compatibility class. For Type II networks, the set of stoichiometric generators is not empty. The stoichiometric generators define subnetworks, such that if these subnetworks possess multiple steady states, then the original network also allows multiple steady states~\cite{conradi2007subnetwork}. 

In addition, we show that networks of different types cannot be dynamically equivalent. Theorem~\ref{lem:unique of deficiency_one} establishes this fact, and this implies that any mass-action system, at most, has one type of weakly reversible deficiency one realization, either Type I or Type II. In Section~\ref{sec:pointed_cone} we analyze in depth the extreme vectors of the cone $\text{ker}(\bW)\cap\mathbb{R}^m_{\geq 0}$ for weakly reversible deficiency one networks. In particular, we show that for Type I networks with $\ell$ linkage classes, there exist $\ell +1$ generators, while for Type II networks with $\ell$ linkage classes, there exist $\ell +2$ generators. Lemmas~\ref{lem:deficiency_one_cases part2} and~\ref{lem:deficiency_one_cases part3} establish these facts. 

In Section~\ref{sec:algorithms} we describe our main result: the construction and the proof of correctness of Algorithm \ref{algorithm:WR_def_one}.
This algorithm takes as input a matrix of source vertices and the corresponding matrix of net reaction vectors.
Algorithm \ref{algorithm:WR_def_one} uses
\textbf{Alg-WR}$^{\ell=1}$
as a subroutine and determines whether or not there exists a weakly reversible deficiency one realization for this input. 
It is interesting to put this algorithm in the context of existing algorithms in the literature. There has been seminal work in this direction~\cite{johnston2013computing,liptak2016kinetic,szederkenyi2013optimization,rudan2014efficiently,rudan2014polynomial,szederkenyiweak2011finding,liptak2015computing}
based mostly on optimization methods that rely  on mixed integer linear programming to determine the existence of realizations of a certain type. 
The algorithm in this paper uses a novel and straightforward geometric approach by focusing on the extreme vectors of the cone $\text{ker}(\bW)\cap\mathbb{R}^m_{\geq 0}$, instead of posing it as a constrained optimization problem. 
Algorithm \ref{algorithm:WR_def_one} uses 
\textbf{Alg-WR}$^{\ell=1}$
%
and the properties of the extreme vectors of the cone $\text{ker}(\bW)\cap\mathbb{R}^m_{\geq 0}$ to determine the existence of weakly reversible deficiency one realizations. This geometric approach in both algorithms allows for a fully self-contained mathematical analysis of the correctness of these algorithms.

This work opens up interesting new avenues for future research. In particular, the relationship between the minimal set of generators of the cone $\text{ker}(\bW)\cap\mathbb{R}^m_{\geq 0}$ and the deficiency of the network can be explored in greater depth. 
One could also explore the existence of mutually exclusive types of weakly reversible realizations for networks of higher deficiency. 
Another possible direction would be to explore the geometry of this minimal set of generators for weakly reversible networks of higher deficiency.

\bibliographystyle{unsrt}
\bibliography{Bibliography}

\end{document}